\documentclass[11pt]{amsart}
\usepackage{fouriernc} 

\usepackage{Definitions}    
\usepackage{PageSetup}      
\usepackage{Environments}   
\usepackage{array}
\usepackage{graphicx}
\usepackage{comment}

\usepackage[shortlabels]{enumitem}

\newcounter{generators}
\newcounter{relations}
\newcounter{objects}
\newcommand{\comp}[2]{\left({#1}\right)_{#2}}
\newcommand{\catcl}[1]{\mathrm{Cl}({#1})}

\newcommand{\doc}{\mathbf{D}}
\newcommand{\docf}{{D}}

\newcommand{\lffc}[1]{\left[#1\right]^{\mathrm{X}}} 
\newcommand{\lff}[1]{\Phi_{#1}^{\mathrm{X}}}

\newcommand{\nlffc}[1]{\left[#1\right]^{\mathrm{N}}} 
\newcommand{\nlff}[1]{\Phi_{#1}^{\mathrm{N}}}

\newcommand{\pffc}[1]{\left[#1\right]^{\mathrm{P}}} 
\newcommand{\pff}[1]{\Phi_{#1}^{\mathrm{P}}} 

\newcommand{\lsff}[1]{\Phi_{#1}^{\mathrm{LS}}}
\newcommand{\lsffc}[1]{\left[#1\right]^{\mathrm{LS}}}

\newcommand{\fols}[1]{\vert{#1}\vert}

\newcommand{\nlsff}[1]{\Phi_{#1}^{\mathrm{NS}}} 
\newcommand{\nlsffc}[1]{\left[#1\right]^{\mathrm{NS}}}

\newcommand{\fobic}[1]{\vert{#1}\vert} 
\newcommand{\frbic}[1]{\left[{#1}\right]^{\mathrm{B}}}
\newcommand{\mfrbic}[1]{\left[{#1}\right]^{\mathrm{P}}}

\newcommand{\prelimmfrbic}[1]{\left[{#1}\right]^{\mathrm{O}}}

\newcommand{\fmfrbic}[1]{\Phi_{#1}^{\mathrm{P}}}

\newcommand{\fosmc}[1]{\vert{#1}\vert}  
\newcommand{\frsmc}[1]{\left[{#1}\right]^{\mathrm{S}}}

\newcommand{\fosbi}[1]{\vert{#1}\vert}  
\newcommand{\frsbit}[1]{\left[{#1}\right]^{\mathrm{Sh}}}
\newcommand{\shfrsbit}[1]{\left[{#1}\right]^{\mathrm{LSh}}}
\newcommand{\shlff}[1]{H_{#1}^{\mathrm{L}}}
\newcommand{\nshfrsbit}[1]{\left[{#1}\right]^{\mathrm{NSh}}}
\newcommand{\nshlff}[1]{H_{#1}^{\mathrm{N}}}

\newcommand{\gLambda}{G}
\newcommand{\glambda}{X}
\newcommand{\sLambda}{T}
\newcommand{\slambda}{X}

\newcommand{\blacktie}{{ETC}} 
\newcommand{\underlying}{{supporting}}
\newcommand{\punderlying}{{underlying}}

\newcommand{\cspan}[4]{}
\DeclareMathOperator{\oneBicat}{\mathbf{B}icat}
\DeclareMathOperator{\Graphcat}{\mathbf{G}raph}

\DeclareMathOperator{\oneLax}{\mathbf{L}ax}
\DeclareMathOperator{\oneShLax}{\mathbf{L}ax\mathbf{S}h}
\DeclareMathOperator{\oneNSh}{\mathbf{NL}ax\mathbf{S}h}
\DeclareMathOperator{\onenLax}{\mathbf{NL}ax} 
\DeclareMathOperator{\onepLax}{\mathbf{P}seudo} 

\DeclareMathOperator{\oneLMon}{\mathbf{L}ax\mathbf{S}ym\mathbf{M}on}

\DeclareMathOperator{\oneNLMon}{\mathbf{NL}ax\mathbf{S}ym\mathbf{M}on}

\DeclareMathOperator{\oneshBicat}{\mathbf{S}had\mathbf{B}icat}

\DeclareMathOperator{\symcat}{\mathbf{S}ym\mathbf{M}oncat}

\DeclareMathOperator{\setcat}{\mathbf{S}et}
\DeclareMathOperator{\moncat}{\mathbf{M}on}

\usepackage{afterpage}
\usepackage{morefloats}
\usepackage{mathtools}
\usepackage{lscape}
\usepackage{multirow}
\usepackage{enumitem}
\usepackage{graphicx}
\usepackage{wasysym}
\usepackage{scalerel}
\usepackage{stmaryrd}

\usepackage{tabularx}

\usepackage{soul}
\setul{}{1.5pt}

\theoremstyle{theorem}

\colorlet{mylight}{green!40!white}
\colorlet{mymed}{red!60!white}
\colorlet{mydark}{blue!80!white}
\colorlet{mylightfill}{green!20!white}
\colorlet{mymedfill}{red!30!white}
\colorlet{mydarkfill}{blue!40!white}

\colorlet{mydarkred}{red!80!black}
\colorlet{mydarkblue}{blue!70!black}
\colorlet{mydarkgreen}{green!60!black}

\usepackage[all]{xy}
\usepackage{tikz, tikz-3dplot, tikz-cd}
\usetikzlibrary{arrows.meta}
\usetikzlibrary{decorations.markings}
\usetikzlibrary{decorations.pathreplacing,angles,quotes}

\usetikzlibrary{backgrounds}
\tdplotsetmaincoords{70}{120}
\usetikzlibrary{patterns}

\newcommand{\xto}{\xrightarrow}

\usepackage{tikz}
\usepackage{subcaption}

\newcommand{\ra}{\longrightarrow}

\makeatletter
\def\slashedarrowfill@#1#2#3#4#5{%
  $\m@th\thickmuskip0mu\medmuskip\thickmuskip\thinmuskip\thickmuskip
   \relax#5#1\mkern-7mu%
   \cleaders\hbox{$#5\mkern-2mu#2\mkern-2mu$}\hfill
   \mathclap{#3}\mathclap{#2}%
   \cleaders\hbox{$#5\mkern-2mu#2\mkern-2mu$}\hfill
   \mkern-7mu#4$%
}
\def\rightslashedarrowfill@{%
  \slashedarrowfill@\relbar\relbar\mapstochar\rightarrow}
\newcommand\xslashedrightarrow[2][]{%
  \ext@arrow 0055{\rightslashedarrowfill@}{#1}{#2}}
\makeatother

\usepackage{braket}

\renewcommand{\Set}{\ensuremath{\mathbf{Set}}}

\newcommand{\sh}[1]{{\ensuremath{\hspace{1mm}\makebox[-1mm]{$\langle$}\makebox[0mm]{$\langle$}\hspace{1mm}{#1}\makebox[1mm]{$\rangle$}\makebox[0mm]{$\rangle$}}}}

\newcommand{\bigsh}[1]{{\ensuremath{\hspace{1mm}\makebox[-1mm]{$\big\langle$}\makebox[0mm]{$\big\langle$}\hspace{1mm}{#1}\makebox[1mm]{$\big\rangle$}\makebox[0mm]{$\big\rangle$}}}}

\newcommand{\odots}[1]{\odot}

\usetikzlibrary{shapes,snakes}
\tikzset{pb/.style={draw,regular polygon,regular polygon sides=3,inner sep=0pt,shape border rotate=180,font=\scriptsize}}
\tikzset{pbflat/.style={draw,regular polygon,regular polygon sides=3,shape border rotate=270,inner sep=1pt,font=\scriptsize, fill=white}}

\usepackage{xkeyval}
\makeatletter

\define@cmdkey      [PRE] {fam}     {A}  {}
\define@cmdkey      [PRE] {fam}     {B}  {}
\define@cmdkey      [PRE] {fam}     {C}  {}
\define@cmdkey      [PRE] {fam}     {D}  {}
\define@cmdkey      [PRE] {fam}     {f}  {}
\define@cmdkey      [PRE] {fam}     {g}  {}
\define@cmdkey      [PRE] {fam}     {h}  {}
\define@cmdkey      [PRE] {fam}     {k}  {}
\define@cmdkey      [PRE] {fam}     {E}  {}
\define@cmdkey      [PRE] {fam}     {F}  {}
\define@cmdkey      [PRE] {fam}     {G}  {}
\define@cmdkey      [PRE] {fam}     {H}  {}
\define@cmdkey      [PRE] {fam}     {phi}  {}
\define@cmdkey      [PRE] {fam}     {gamma}  {}
\define@cmdkey      [PRE] {fam}     {eta}  {}
\define@cmdkey      [PRE] {fam}     {kappa}  {}
\define@cmdkey      [PRE] {fam}     {a}  {}
\define@cmdkey      [PRE] {fam}     {b} {}
\define@cmdkey      [PRE] {fam}     {c} {}
\define@cmdkey      [PRE] {fam}     {d} {}
\define@cmdkey      [PRE] {fam}     {wi} {}
\define@cmdkey      [PRE] {fam}     {he} {}

\presetkeys         [PRE] {fam} {
	A = ,
	B = ,
	C = ,
	D = ,
	f = ,
	g = ,
	h = ,
	k = ,
	E = ,
	F = ,
	G = ,
	H = ,
	phi = ,
	gamma = ,
	eta = ,
	kappa = ,
	a = ,
	b = ,
	c = ,
	d = ,
	wi = ,
	he = ,
	}{}

\newcommand{\commutingcube}[1]{
	\setkeys[PRE]{fam}{#1}
\begin{tikzcd}[ampersand replacement=\&,column sep={\cmdPRE@fam@wi},row sep={\cmdPRE@fam@he}]
			\& \cmdPRE@fam@B \ar[ddd, near end, "\cmdPRE@fam@b"]\ar[rr, "\cmdPRE@fam@g"]	\&\& \cmdPRE@fam@D \ar[ddd, "\cmdPRE@fam@d"] 
				\\
	\cmdPRE@fam@A \ar[ddd, "\cmdPRE@fam@a"]\ar[rr,near start,  crossing over,  "\cmdPRE@fam@h"]\ar[ru, "\cmdPRE@fam@f"]	\&\& \cmdPRE@fam@C \ar[ru, "\cmdPRE@fam@k"]	
				\\ \\
			{}\& \cmdPRE@fam@F \ar[rr, near start, "\cmdPRE@fam@gamma"]		\&\& \cmdPRE@fam@H
				 \\
	\cmdPRE@fam@E \ar[rr, "\cmdPRE@fam@eta"]\ar[ru, "\cmdPRE@fam@phi"]		\&\& \cmdPRE@fam@G \ar[ru, "\cmdPRE@fam@kappa"]\arrow [from=uuu,  near start,crossing over,  "\cmdPRE@fam@c"]
\end{tikzcd}
}

\newcommand{\commutingsquare}[1]{
	\setkeys[PRE]{fam}{#1}
\begin{tikzcd}[ampersand replacement=\&,column sep={\cmdPRE@fam@wi},row sep={\cmdPRE@fam@he}]
					\& \cmdPRE@fam@B \ar[rr,"\cmdPRE@fam@g"]			\&	\& \cmdPRE@fam@D  \\
		\cmdPRE@fam@A \ar[rr,"\cmdPRE@fam@h"]\ar[ru,"\cmdPRE@fam@f"]	\&\& \cmdPRE@fam@C \ar[ru,"\cmdPRE@fam@k"]	\&
\end{tikzcd}
}

\newcommand{\triangularprism}[1]{
	\setkeys[PRE]{fam}{#1}
\begin{tikzcd}[ampersand replacement=\&,column sep={\cmdPRE@fam@wi},row sep={\cmdPRE@fam@he}]
		\& \cmdPRE@fam@B \ar[dd, near start, "\cmdPRE@fam@b"]\ar[rd,"\cmdPRE@fam@g"]
		\\
	\cmdPRE@fam@A \ar[dd,"\cmdPRE@fam@a"]\ar[rr, crossing over, near start,  "\cmdPRE@fam@h"]\ar[ru,"\cmdPRE@fam@f"] 
		\&\& \cmdPRE@fam@C \ar[dd,"\cmdPRE@fam@c"] 
		\\
		\& \cmdPRE@fam@F \ar[rd,"\cmdPRE@fam@gamma"] 
		\\
		\cmdPRE@fam@E \ar[rr,"\cmdPRE@fam@eta"]\ar[ru,"\cmdPRE@fam@phi"] \&\& \cmdPRE@fam@G
\end{tikzcd}
}

\newcommand{\triangularprismtwo}[1]{
	\setkeys[PRE]{fam}{#1}
	
\begin{tikzcd}[ampersand replacement=\&,column sep={\cmdPRE@fam@wi},row sep={\cmdPRE@fam@he}]
	\cmdPRE@fam@B \ar[dd,"\cmdPRE@fam@b"']\ar[rr,"\cmdPRE@fam@g"]\ar[rd,"\cmdPRE@fam@h"'] 
		\&\& \cmdPRE@fam@D \ar[dd,"\cmdPRE@fam@d"] 
		\\
	\& \cmdPRE@fam@C \ar[ru, "\cmdPRE@fam@k"']	
		 \\
	\cmdPRE@fam@F \ar[rr, near start, "\cmdPRE@fam@gamma"]\ar[rd,"\cmdPRE@fam@eta"'] 
		\&\& \cmdPRE@fam@H
		 \\
	\& \cmdPRE@fam@G \ar[ru,"\cmdPRE@fam@kappa"']\arrow [from =uu, near start,  crossing over, "\cmdPRE@fam@c"]
\end{tikzcd}
	
}

\makeatother

\usepackage{tikz-cd} 
\usepackage{multirow}

\hypersetup{
	pdfkeywords={latex starter,template},
	pdfauthor={Niles Johnson},
}



\title{Coherence for bicategories, lax functors, and shadows}
\author{Cary Malkiewich} 
\email{malkiewich@math.binghamton.edu}
\address{Binghamton University, PO Box 6000, Binghamton, NY 13902}
\author{Kate Ponto}
\email{kate.ponto@uky.edu}
\address{University of Kentucky, 719 Patterson Office Tower, Lexington, KY 40506}
\date{2nd Sept, 2021}

\begin{document}

\begin{abstract}
Coherence theorems are fundamental to how we think about monoidal categories and their generalizations.  In this paper we revisit Mac Lane's original proof of coherence for monoidal categories using the Grothendieck construction. This perspective makes the approach of Mac Lane's proof very amenable to generalization. We use the technique to give efficient proofs of  many standard coherence theorems and new coherence results for bicategories with shadow and for their functors.
\end{abstract}
 
 \maketitle
 
 \setcounter{tocdepth}{1}
 \tableofcontents

\section{Introduction}
Colloquially, Mac Lane's coherence theorem for monoidal categories says ``all diagrams that should commute do commute''.  For a more formal statement, recall that there is a forgetful functor from the category of monoidal categories and strict monoidal functors to the category of sets
\[\moncat\to\setcat,\]
taking each category to its underlying set of objects.\footnote{As usual, if the categories in question are large then their underlying ``sets'' of objects will be large. This can be resolved either by expanding the universe when defining $\setcat$, or by restricting to small categories. Since our goal is to prove that diagrams commute, this always  reduces to the case of small diagrams.
}
This has a left adjoint free functor.  A diagram in a monoidal category $\sC$ is {\bf formal} if it lifts to the free monoidal category on the underlying set of objects of $\sC$. In other words, if it lifts against the counit of the above adjunction.
\begin{thm}[\cite{maclane_orginial_coherence,maclane}]\label{intro:monoidal_coherence}
All formal diagrams in a monoidal category commute.
\end{thm}

For certain other kinds of categories and functors the same result holds. 
\begin{thm}\label{intro:simple_coherence}
All formal diagrams in the categorical structures in \cref{table:intro_theorems_coherent} commute.
\end{thm}
\begin{table}[h]
	\begin{tabular}{rcc}
		{\bf Categorical structure}  & {\bf References} \\\hline
		monoidal categories  & \cite{maclane,maclane_orginial_coherence,power}&\cref{bicat_coherence_formal} \\
		strong monoidal functors & \cite{power} &\cref{strong_coherence}\\
		normal lax monoidal functors  & &\cref{strong_coherence_normal_lax} \\\hline
		bicategories  & \cite{power,joyal_street} &\cref{bicat_coherence_formal}\\
		pseudofunctors  & \cite{power,joyal_street} &\cref{strong_coherence}\\
		normal lax functors  &&\cref{strong_coherence_normal_lax} \\ \hline
	\end{tabular}
	\vspace{1em}
	\caption{Categorical structures for \cref{intro:simple_coherence}}\label{table:intro_theorems_coherent}
\end{table}

As we add symmetry and move to more general kinds of functors, the situation gets more complicated.  
For example, in a symmetric monoidal category, it would be unreasonable to expect a formal diagram to commute if two parallel composites in the diagram induced different permutations on the objects. So we add this to the hypotheses of \cref{intro:monoidal_coherence,intro:simple_coherence}. We say a formal diagram in a symmetric monoidal category is {\bf expected to commute}
 ({\blacktie}) if every pair of parallel composites induces the same  permutation.

For the categorical structures in \cref{table:intro_theorems}, we replace the symmetric group 
by the group (or category) in the middle column, and then define ``{\blacktie}'' similarly.  With that modification, we have the following result.
\begin{thm}\label{intro:black_tie_coherence}
All {\blacktie}  diagrams in the categorical structures in \cref{table:intro_theorems} commute.
\end{thm}
\begin{table}[h]
	\begin{tabular}{>{\raggedleft\arraybackslash}p{5.5cm}ccc}
		{\bf Categorical structure} & {\bf Index} & {\bf References} \\\hline
		lax monoidal functors &  $\mathbf\Delta$ & \cite{epstein,lewis_thesis}&\cref{lax_not_coherence} \\\hline
		lax functors of bicategories & $\mathbf\Delta$ & &\cref{lax_not_coherence}\\\hline
		symmetric monoidal categories & $\Sigma_n$& \cite{maclane,maclane_orginial_coherence,joyal_street} &\cref{symcat_coherence_formal}\\
		strong sym. monoidal functors & $\Sigma_n$& \cite{joyal_street} &\cref{thm:coherence_strong_symmetric}\\
		normal lax sym. mon. functors & $\Sigma_n$ &&\cref{thm:coherence_normal_symmetric}  \\
		lax sym. mon. functors & $\Sigma_n$ \& $\mathbf{Fin}$ & \cite{lewis_thesis} &\cref{symmetric_not_coherence}\\\hline
		shadowed bicategories & $C_n$&&\cref{shbicat_coherence_formal}\\
		strong shadow functors & $C_n$&&\cref{thm:coherence_strong_shadow} \\
		normal lax shadow functors & $C_n$&&\cref{thm:coherence_normal_shadow}  \\
		lax shadow functors & $C_n$ \& $\mathbf\Lambda$ &&\cref{shadow_not_coherence} \\\hline
	\end{tabular}
	\vspace{1em}
	\caption{Categorical structures for \cref{intro:black_tie_coherence}}\label{table:intro_theorems}
\end{table}

Our proofs of \cref{intro:simple_coherence,intro:black_tie_coherence} are combinatorial and follow the spirit of Mac Lane's original proof.  They are  closely related to the approaches in   \cite{epstein,kelly_maclane,lewis_thesis}.  (They are  less similar to the strictification results in \cite{power,joyal_street,GPS_tricat} -- these results are far-reaching, but they don't apply to lax and normal lax functors.)  The fundamental insight is that formal diagrams in the categorical structures in \cref{table:intro_theorems_coherent,table:intro_theorems}  can be built as a series of Grothendieck constructions (\cref{defn:Grothendieck_construction}) starting from very small pieces.   As an example, we first build formal diagrams for associators in a bicategory using a Grothendieck construction, then we add in unitor maps with a second Grothendieck construction.  To build formal diagrams in a shadowed bicategory, we use a third Grothendieck construction  to add in rotator maps.

\subsection*{Outline} In \cref{cliques_language} we will recall the definitions of cliques and the Grothendieck construction that are the fundamental building blocks of the proofs of \cref{intro:simple_coherence,intro:black_tie_coherence}.  In \cref{sec:presentations} we recall the combinatorial ``generators and relations'' presentations of the categories in \cref{table:intro_theorems}. In \cref{sec:coherence_cat} we prove the coherence theorems for bicategories, symmetric monoidal categories, and shadowed bicategories.  In \cref{sec:coherence_functor} we prove the corresponding results for functors.

\subsection*{Acknowledgments}
The authors are pleased to acknowledge contributions to this project that emerged from enjoyable conversations with Mike Shulman and Ross Street.

CM was supported by the NSF grants DMS-2005524 and DMS-2052923. KP was supported by NSF grants DMS-1810779 and DMS-2052923, and the Royster research professorship at the University of Kentucky.

\section{Diagrams of cliques}\label{cliques_language}
The proofs of \cref{intro:simple_coherence,intro:black_tie_coherence} follow an identical structure, which we set up in this section.  
\begin{defn}A category $\bC$ is {\bf thin} if, for each ordered pair of objects $a,b$ in $\bC$, the set of morphisms $\bC(a,b)$ contains at most one element.
\end{defn}
In a thin category all diagrams commute.

\begin{defn} A (small) category $\bK$ is an  \textbf{abstract clique} if it satisfies any of the following equivalent conditions:
\begin{itemize}
	\item $\bK$ is a nonempty connected thin groupoid.
	\item $\bK$ is contractible (equivalent to the one-point category).
	\item $\bK$ has nonempty object set, and for each ordered pair of objects $a,b$ in $\bK$, the set $\bK(a,b)$ has precisely one element.
\end{itemize}
For any category $\bC$, a \textbf{clique} in $\bC$ is an abstract clique $\bK$ and a functor $K\colon \bK \to \bC$. If $K$ is the inclusion of a subcategory then we simply say $\bK \subseteq \bC$ is a clique.
\end{defn}
We think of  cliques in $\bC$ as ``thick objects'' -- objects defined up to canonical isomorphism. For a clique $(\bK,K)$ in $\bC$, the objects $K(k)$ for $k\in \bK$ are {\bf models} or {\bf representatives} of $(\bK,K)$.  The maps in the image of $K$ are {\bf canonical isomorphisms}.

\begin{example}\label{ex:define_odot_clique}
	Let $\sB$ be a bicategory. The coherence theorem for bicategories (\cref{bicat_coherence}) implies that each ordered tuple of 1-cells $X_i \in \sB(A_{i-1},A_i)$ defines a clique 
\[\bigodot_{i=1}^nX_i\]
in the category $\sB(A_0,A_n)$.  
The objects 
are pairs consisting of 
\begin{enumerate}
\item an ordered tuple of nonnegative integers $(j_0,j_1,\ldots j_n)$ and 
\item a
parenthesization of the expression
\begin{equation}\label{eq:lots_of_units_and_xs}\underbrace{I\odot I\odot \cdots \odot  I}_{j_0}\odot X_1\odot\underbrace{I\odot I\odot\cdots \odot I}_{j_1}\odot X_2\odot\ldots\odot X_{n-1}\underbrace{I\odot I\odot\cdots \odot I}_{j_{n-1}}\odot X_n\odot \underbrace{I\odot I\odot \cdots \odot I}_{j_n}.
\end{equation}
\end{enumerate}
The morphisms are generated by the associator and unitor maps. Note that there are no maps between the $X_i$.
\end{example}

\begin{example}[Generalization of  \cref{ex:define_odot_clique}]\label{ex:define_odot_clique_F}
Let $F\colon \sB\to \sB'$ be a lax functor of bicategories. An ordered tuple of 1-cells $X_i \in \sB(A_{i-1},A_i)$ and a totally ordered map $\alpha\colon \{1,...,n\} \to \{1,...,k\}$ defines 
a clique we denote
\[ \bigodot_{j \in \underline{k}} F\left( \bigodot_{i \in \alpha^{-1}(j)} X_i \right). \]
The objects are 
\[\ob\left(\bigodot_{j \in \underline{k}}Y_j\right)\times \prod_{j\in\underline{k}}\ob \left( \bigodot_{i \in \alpha^{-1}(j)} X_i \right) \]
(We think of the $Y_j$ as placeholders for the terms $F\left( \bigodot_{i \in \alpha^{-1}(j)} X_i \right)$.)
This defines a 1-cell in $\sB'$, by first adding units and composing to give the desired model for each $\bigodot_{i \in \alpha^{-1}(j)} X_i$, then applying $F$ to each of these, and finally adding units and composing along the model for $\bigodot_{j \in \underline{k}}Y_j$. A typical example of such a 1-cell is
\[ (F((X_1 \odot (I \odot X_2)) \odot I)\odot F(I)) \odot (I \odot F(X_3)). \]

The morphisms are generated by the unit and associator maps for $\sB$ and $\sB'$.
These give well-defined isomorphisms in $\sB'$ since the morphisms for the outside product $\bigodot_{j \in \underline{k}} (-)$ are natural with respect to maps of the inside products. 

Since $ \bigodot_{j \in \underline{k}} F\left( \bigodot_{i \in \alpha^{-1}(j)} X_i \right)$ is a product of cliques, it is a clique.
\end{example}

\begin{rmk}
	Kelly's notion of a \textbf{club} \cite{kelly_clubs} formalizes the constructions present in the previous example, specifically, the way one can form models for a big tensor product by composing models for the tensor products $\bigodot_{i \in \alpha^{-1}(j)} X_i$ with a model for $\bigodot_j Y_j$.
\end{rmk}

\begin{defn}\label{map_of_cliques}
A \textbf{map of cliques} $(\bA,A) \to (\bB,B)$ is a collection of maps 
\[\{A(a) \to B(b) \in \bC\}_{(a,b)\in \ob (\bA\times \bB)}\] 
so that the following square commutes for all maps $f\in \bA$ and $g\in \bB$:
\[\xymatrix @R=1.5em{A(a) \ar[r]\ar[d]^-{A(f)}& B(b)\ar[d]^{B(g)}
\\A(a') \ar[r]& B(b')
}\]
Informally, it maps each object $A(a)$ to each object $B(b)$ in a way that commutes with all of the canonical isomorphisms.
\end{defn}

\begin{rmk}\label{rmk:thick_maps_partially_defined}
Any nonempty collection of pairs $S \subseteq \ob(\bA \times \bB)$ and a collection of maps 
\[\{A(a) \to B(b) \in \bC\}_{(a,b)\in S}\] commuting with the canonical isomorphisms extends in a unique way to a clique map $(\bA,A) \to (\bB,B)$.  If we define a map this way, we call the elements in $S$ the {\bf admissible models} for this map of cliques.

\end{rmk}

\begin{example}\label{ex:tensor_product_maps}
In a monoidal category or bicategory, it is common to define maps between tensor products
\begin{equation}
\label{eq:simple_adm_model_example}
X_1 \otimes X_2 \otimes X_3 \to Y_1 \otimes Y_2
\end{equation}
by defining a collection of maps on smaller products, such as
\[ f\colon X_1 \otimes X_2 \to Y_1, \qquad g\colon X_3 \to Y_2. \]
Formally, the expression $X_1 \otimes X_2 \otimes X_3$ denotes a clique, and that clique has a nonempty subset of models for the product in which $f_1$ and $f_2$ can be applied. In particular, for the model $(I \otimes (X_1 \otimes X_2)) \otimes X_3$ we can define the desired map as
\[ \xymatrix @C=4em{ (I \otimes (X_1 \otimes X_2)) \otimes X_3 \ar[r]^-{(1 \otimes f) \otimes g} & (I \otimes Y_1) \otimes Y_2, } \]
but the model $X_1 \otimes (X_2 \otimes X_3)$ does not admit such an easy definition because $X_1$ and $X_2$ are not grouped together.

The point of \cref{rmk:thick_maps_partially_defined} is that we only have to define the map on \emph{some} models for the product. We define it on those models where an $X_1 \otimes X_2$ somewhere in the word for $X_1 \otimes X_2 \otimes X_3$, mapping to the corresponding model for $Y_1 \otimes Y_2$, as above. We then check it commutes with the canonical isomorphisms \emph{between the admissible models}, which is easy. In summary, we get \eqref{eq:simple_adm_model_example} defined on the entire clique, but we only had to explicitly define it on the models where the definition is easy.
\end{example}

\begin{defn}\label{defn:Grothendieck_construction}
For a category $\bI$ and abstract cliques $\{{\doc}(i)\}_{i\in \ob\bI}$,
the  {\bf Grothendieck construction} on the ${\doc}(i)$, denoted  
$\int_\bI \doc$, is the category with
\begin{itemize}
\item objects the
pairs $(i,x)$ with $i \in \ob \bI$ and a $x \in {\doc}(i)$ and 
\item a morphism $(i,x) \to (j,y)$ for each
morphism $i \to j$ in $\bI$.
\end{itemize}
\end{defn}
Note that each clique $ \doc(i)$ includes into $\int_\bI \doc$ as the objects $x \in  \doc(i)$ and the morphisms $(i,x) \to (i,y)$ corresponding to the identity map $i\to i$.

\begin{lem}\label{lem:consequences_of_projection_1} Forgetting the elements of $ {\doc}(i)$ defines an equivalence of categories \[\pi\colon \int_\bI \doc\to \bI.\] 
\end{lem}

\begin{proof}
This functor is surjective since the object sets of ${\doc}(i)$ are nonempty, and fully faithful since each  $ {\doc}(i)$ is a clique.
\end{proof}

\begin{cor}\label{lem:consequences_of_projection_2} $\int_\bI \doc$ is thin or an abstract clique precisely when $\bI$ is thin or an abstract clique, respectively. 
\end{cor}

Clique maps (\cref{map_of_cliques}) can be composed, and their compositions are equal if and only if they are equal on a single representative.  Cliques and morphisms of cliques in a category $\bC$ form a category we denote $\catcl{\bC}$.  Note that $\catcl{\bC}$ is equivalent to $\bC$.

We call a diagram $\docf\colon \bI \to \catcl{\bC}$  a {\bf diagram of cliques} in $\bC$. The image of each $i\in \bI$ is a pair 
\[({\doc}(i),\docf(i)\colon {\doc}(i) \to \bC)\]
 consisting of an abstract clique ${\doc}(i)$ and a clique in $\bC$, $\docf(i)\colon {\doc}(i) \to \bC$.

\begin{lem}\label{clique_bijection}
	For fixed $\bI$, there is a bijection between diagrams of cliques $\docf\colon \bI \to \catcl{\bC}$  and pairs consisting of a collection of abstract cliques $\{{\doc}(i)\}_{i\in I}$ and a functor $\int_\bI \doc \to \bC$.
\end{lem}

\begin{proof}
	Given a diagram of cliques, we define $\int_\bI \doc \to \bC$ by sending each morphism $(i,x) \to (j,y)$ to the canonical map $\docf(i)(x) \to \docf(j)(y)$ given by $\docf$. This respects identity and composition since these operations for cliques respect the restriction to one representative.
	
	Conversely, given a diagram $\int_\bI \doc \to \bC$, we define a diagram of cliques by sending each $i \to j$ to the map of cliques $\docf(i) \to \docf(j)$ that for each pair of objects $x,y$ applies the morphism $(i,x) \to (j,y)$. This is well-defined since composing this with isomorphisms $x \cong x'$ and $y \cong y'$ gives the corresponding morphism $(i,x') \to (j,y')$ from our diagram. It respects identity and composition, again by restricting to any one representative in each clique.
\end{proof}

\begin{rmk}\label{rmk:plan_of_proofs}
The main results of this paper all amount showing that some category of interest $\bC$ is equivalent to  an easier to understand category $\bI$.  The technique is:
\begin{enumerate}
\item\label{rmk:plan_of_proofs_1} construct a diagram of cliques $\docf\colon \bI \to \catcl{\bC}$, 
\item\label{rmk:plan_of_proofs_2}  apply \cref{clique_bijection} to define a functor 
\begin{equation}\label{eq:step_in_the_only_plan}
\int_\bI \doc \to \bC,
\end{equation}
\item\label{rmk:plan_of_proofs_3} verify that the map \eqref{eq:step_in_the_only_plan} is an isomorphism of categories, and 
\item\label{rmk:plan_of_proofs_4} use \cref{lem:consequences_of_projection_1} to conclude  there is an equivalence $\bI\to \bC$.
\end{enumerate}
\end{rmk}

\section{Presentations of categories}\label{sec:presentations}
The proofs of \cref{intro:simple_coherence,intro:black_tie_coherence} are combinatorial, so we will need explicit descriptions of the indexing categories from \cref{table:intro_theorems} (which play the role of $\bI$ in \cref{rmk:plan_of_proofs}\ref{rmk:plan_of_proofs_1}).  We give those descriptions in this section.
A reader who is not especially fascinated by presentations of categories is free to skip this section and refer back to it as needed.

A {\bf presentation} of a category consists of a collection of objects $a, b, \ldots$, generating morphisms $a \to b$, and a collection of relations, each of which says that two different words in the generators $a \rightrightarrows c$ are equal to each other. An {\bf invertible generator} $a \leftrightarrow b$ is a pair of generators $a \to b$ and $b \to a$, together with two relations making them into inverses of each other.

\begin{present}\label{present_e_cyclic}
 Let $C_n$ denote the cyclic group of order $n$. Let $\sB C_n$ be the category with a single object with endomorphisms $C_n$. There is a presentation of $\sB C_n$ with generators 
	\begin{enumerate}[start=1,label={\bfseries G\arabic{generators}}]
\item\stepcounter{generators} \label{item:cyclic_generator}  $a_k$ for  $0 \leq k < n $
	\end{enumerate}
and relations	
	\begin{enumerate}[start=1,label={\bfseries R\arabic{relations}}]
\item\stepcounter{relations}\label{item:cyclic_relation}  $\ a_ka_l = a_{k+l}$, indices mod $n$.
\end{enumerate}
The generators could either be taken to be ordinary generators, or invertible generators. If  $k=l=0$, \ref{item:cyclic_relation}  becomes $a_0a_0 = a_0$. This  is equivalent to  $a_0 = 1$, so we use $a_0=1$ instead. 
\end{present}

\begin{present}\label{presentation:symmetric_group}
Let $\Sigma_n$ denote the symmetric group on $n$ letters and $\sB\Sigma_n$ the corresponding one-object category. There is a presentation of $\sB\Sigma_n$ with generators
\begin{enumerate}[start=1,label={\bfseries G\arabic{generators}}]
\item\stepcounter{generators}\label{gen:transposition}  adjacent transpositions 
$\tau_i = (i \ i+1),$ for  $1 \leq i < n$
\end{enumerate}
and relations 
\begin{enumerate}[start=1,label={\bfseries R\arabic{relations}}]
\item\stepcounter{relations} \label{squaretozero}
$\tau_i^2 = 1$ for  $1 \leq i < n$
\item\stepcounter{relations}  \label{farawaycommute}
$\tau_i\tau_j = \tau_j\tau_i$ for $ |i-j| > 1 $, and  
\item\stepcounter{relations}  
\label{threecycle}
$(\tau_i\tau_{i+1})^3 = 1$ for $1 \leq i < n-1$.
\end{enumerate}

\end{present}

Let $\mathbf\Delta$ be a skeleton of the category of finite totally ordered sets. We allow the sets to be empty, and we label their elements starting with 1, so the objects of $\mathbf\Delta$ are
\[ \underline{0} = \emptyset, \quad \underline{1} = \{1\}, \quad \underline{2} = \{1,2\}, \quad \underline{3} = \{1,2,3\}, \quad \textup{etc.} \]
This is the simplex category but with objects relabeled $\underline{n} = [n-1]$ and with an extra object for the empty set.

\begin{present}\label{delta_presentation}
	There is a  presentation of $\Delta$ 
	with generators 
\begin{enumerate}[start=1,label={\bfseries G\arabic{generators}}]
\item\stepcounter{generators}\label{gen:coface} {\bf coface} maps
\[ d^i: \underline{n-1} \ra \underline{n}, \quad 1 \leq i \leq n \]
\[ d^i(j) = \left\{ \begin{array}{ccc} j &\textup{ if }& j < i \\
j+1 &\textup{ if }& j \geq i \end{array} \right. \]
and 
\item\stepcounter{generators}\label{gen:codegeneracy} {\bf codegeneracy} maps
\[ s^i: \underline{n+1} \ra \underline{n}, \quad 1 \leq i \leq n \]
\[ s^i(j) = \left\{ \begin{array}{ccc} j &\textup{ if }& j \leq i \\
j-1 &\textup{ if }& j > i \end{array} \right. \]
\end{enumerate}

\begin{figure}
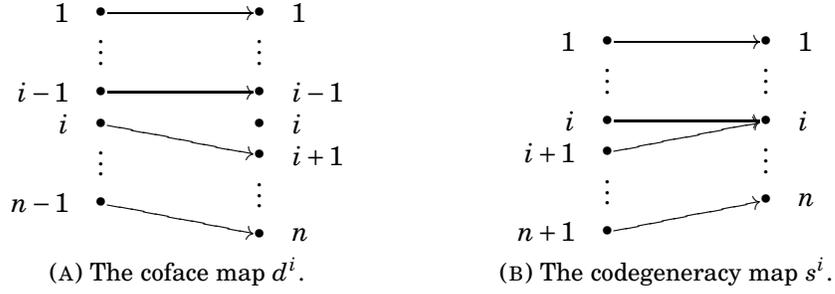

	\hspace{1.5cm}
     \centering
     \begin{subfigure}[b]{0.33\textwidth} 
		\centering $\xy 0;<12pt,0pt>:
	(0,0)*{\bullet}="L1"+(-1,0)*!R{1};
	"L1"+(0,-1)*{\vdots}="Ldots";
	"Ldots"+(0,-1.5)*{\bullet}="Li-1"+(-1,0)*!R{i-1};
	"Li-1"+(0,-1)*{\bullet}="Li"+(-1,0)*!R{i};
	"Li"+(0,-1)*{\vdots}="Ldots2";
	"Ldots2"+(0,-1.5)*{\bullet}="Ln-1"+(-1,0)*!R{n-1};
	(5,0)*{\bullet}="R1"+(1,0)*!L{1};
	"R1"+(0,-1)*{\vdots}="Rdots";
	"Rdots"+(0,-1.5)*{\bullet}="Ri-1"+(1,0)*!L{i-1};
	"Ri-1"+(0,-1)*{\bullet}="Ri"+(1,0)*!L{i};
	"Ri"+(0,-1)*{\bullet}="Ri+1"+(1,0)*!L{i+1};
	"Ri+1"+(0,-1)*{\vdots}="Rdots2";
	"Rdots2"+(0,-1.5)*{\bullet}="Rn"+(1,0)*!L{n};
	"L1";"R1" **\dir{-}; ?>*\dir{>};
	"Li-1";"Ri-1" **\dir{-}; ?>*\dir{>};
	"Li";"Ri+1" **\dir{-}; ?>*\dir{>};
	"Ln-1";"Rn" **\dir{-}; ?>*\dir{>};
	\endxy$
	\caption{The coface map $d^i$.}
     \end{subfigure}
     \hfill
     \begin{subfigure}[b]{0.33\textwidth}
         \centering
	$\xy 0;<12pt,0pt>:
	(0,0)*{\bullet}="L1"+(-1,0)*!R{1};
	"L1"+(0,-1)*{\vdots}="Ldots";
	"Ldots"+(0,-1.5)*{\bullet}="Li"+(-1,0)*!R{i};
	"Li"+(0,-1)*{\bullet}="Li+1"+(-1,0)*!R{i+1};
	"Li+1"+(0,-1)*{\vdots}="Ldots2";
	"Ldots2"+(0,-1.5)*{\bullet}="Ln+1"+(-1,0)*!R{n+1};
	(5,0)*{\bullet}="R1"+(1,0)*!L{1};
	"R1"+(0,-1)*{\vdots}="Rdots";
	"Rdots"+(0,-1.5)*{\bullet}="Ri"+(1,0)*!L{i};
	"Ri"+(0,-1)*{\vdots}="Rdots2";
	"Rdots2"+(0,-1.5)*{\bullet}="Rn"+(1,0)*!L{n};
	"L1";"R1" **\dir{-}; ?>*\dir{>};
	"Li";"Ri" **\dir{-}; ?>*\dir{>};
	"Li+1";"Ri" **\dir{-}; ?>*\dir{>};
	"Ln+1";"Rn" **\dir{-}; ?>*\dir{>};
	\endxy$
	\caption{The codegeneracy map $s^i$.}
     \end{subfigure}
	\hspace{1.5cm}
        \caption{Generators for $\Delta$}
\end{figure}

and relations
\begin{enumerate}[start=1,label={\bfseries R\arabic{relations}}]
\item\stepcounter{relations}  \label{faceface}
$d^i d^j = d^{j+1} d^i $ for $ i \leq j $,
\item\stepcounter{relations}  \label{degdeg}
$s^j s^i = s^i s^{j+1} $ for $ i \leq j$,
\item\stepcounter{relations}  \label{degface1}
$s^j d^i = d^i s^{j-1} $ for $ i < j$,
\item\stepcounter{relations}  \label{degface2}
$s^j d^i = 1 $ for $ i = j,j+1$, and 
\item\stepcounter{relations}  \label{degface3}
$s^j d^i = d^{i-1} s^j $ for $ i > j+1$.
\end{enumerate}
\end{present}

Let $\sI \subseteq \Delta$ be the subcategory of injective totally-ordered maps.

\begin{present}\label{inj_presentation}
	There is a presentation for $\sI$ with generators \ref{gen:coface}  and relations \ref{faceface}.
\end{present}

In a different direction, let $\mathbf{Fin}$ have the same objects as $\Delta$ but all maps of finite sets, not necessarily preserving the total ordering. The automorphism group of each object $\underline{n}$ is the symmetric group $\Sigma_n$. 

\begin{present}\label{fin_presentation}
There is a presentation for $\mathbf{Fin}$ with generators \ref{gen:transposition} to \ref{gen:codegeneracy} and relations \ref{squaretozero} to \ref{degface3},
 
\begin{enumerate}[start=1,label={\bfseries R\arabic{relations}}]
\item\stepcounter{relations}  \label{relation:fin_swap}
	a {\bf swap} relation $\sigma \circ \alpha = \alpha' \circ \sigma'$ for each $\alpha\colon [k] \to [l]$ in $\mathbf\Delta$ and $\sigma \in \Sigma_l$, and
	\item\stepcounter{relations}  \label{relation:fin_coequalizer}  a coequalizer relation $\alpha \circ \sigma = \alpha$ for $\alpha\colon [k] \to [l]$ in $\mathbf\Delta$ and $\sigma \in \Sigma_k$ that is a permutation of each of the sets $\alpha^{-1}(j)$.
\end{enumerate}
\end{present}
In \ref{relation:fin_swap}, the totally ordered map $\alpha'$ is uniquely determined -- there is only one such map that can make the equation true in finite sets. The permutation $\sigma'$, on the other hand, is only determined up to permutations of the fibers of $\alpha$, but the choice doesn't matter in light of \ref{relation:fin_coequalizer}.

\begin{proof}[Proof of \cref{fin_presentation}]
These relations are satisfied by maps of finite sets. To see these relations suffice, take any word in the generators giving a map of finite sets $\alpha\colon [k] \to [l]$. Using \ref{relation:fin_swap}, the word can be simplified to a word in $\Sigma_k$ followed by a word in $\mathbf\Delta$. Using \ref{squaretozero} to \ref{degface3}, these are determined by the resulting pair of morphisms in $\Sigma_k$ and $\mathbf\Delta$. Two such pairs can give the same map of finite sets only when the totally-ordered parts are identical and the permutations differ by a permutation of each of the sets $\alpha^{-1}(j)$. Using \ref{relation:fin_coequalizer}, the word is  uniquely determined by the corresponding map of finite sets.
\end{proof}

Connes' cyclic category $\Lambda$ has the same objects as $\Delta$  but the morphisms are the ``cyclically ordered'' maps. In this paper, we use a bi-augmented variant $\Lambda'$ with an extra initial object $\underline{0}$, corresponding to the empty cyclically ordered set, and an extra terminal object $\underline{*}$.

\begin{present}\label{lambda_presentation}
The objects of $\Lambda'$ are $\{\underline{0},\underline{1},\underline{2},...,\underline{*}\}$.  Generators are 
\ref{gen:coface}, \ref{gen:codegeneracy} and 
\begin{enumerate}[start=1,label={\bfseries G\arabic{generators}}]
\item\stepcounter{generators}  
a {\bf cycle to the left} map $\tau_{(n)}: \underline{n} \to \underline{n}$ for each $n \geq 1$, along with
\item\stepcounter{generators}
a terminal map $t\colon \underline{1} \to \underline{*}$.
\end{enumerate}
The relations are \ref{faceface}-\ref{degface3},
\begin{enumerate}[start=1,label={\bfseries R\arabic{relations}}]
\item\stepcounter{relations}  
\label{cycleface}
$\tau_{(n)} d^i = d^{i-1}\tau_{(n-1)} $ for $ 2 \leq i \leq n $
\item\stepcounter{relations}   
\label{cycleface2}
$\tau_{(n)} d^1 = d^n$
\item\stepcounter{relations}  
\label{cycledeg}
$\tau_{(n)} s^i = s^{i-1}\tau_{(n+1)}$ for $ 2 \leq i \leq n$
\item\stepcounter{relations}   
\label{cycledeg2}
$\tau_{(n)} s^1 = s^{n}(\tau_{(n+1)})^2$,
\item\stepcounter{relations}  
\label{cycletorsion}
$\tau_{(n)}^n = \id$, and
\item\stepcounter{relations}  
\label{terminal1}
$ts^1 = ts^1\tau_{(2)}$.
\end{enumerate}
\end{present}

The full subcategory of $\Lambda$ on the nonempty sets $\{\underline{1},\underline{2},\underline{3},...\}$ agrees with the cyclic category of Connes, see e.g. \cite{connes,bhm}. This is almost a subcategory of $\mathbf{Fin}$, except that there are $n$ different cyclically ordered maps $\underline{n} \to \underline{1}$ but only one map of finite sets. The relation \ref{terminal1} is sufficient to ensure that every object $\underline{n}$ has a unique map to $\underline{*}$, which factors through $\underline{1}$.

In addition to the presentations of specific categories above, we also need presentations for new categories defined in terms of old categories.
Suppose $\bC$ and $\bD$ are categories with given presentations.

\begin{lem}\label{product_presentation}
	The product $\bC \times \bD$ has a presentation with generators 
\begin{enumerate}[start=1,label={\bfseries G\arabic{generators}}]
\item\stepcounter{generators}  \label{product_presentation_1}
			$(a,x) \to (b,x)$ for each generator $a \to b$ in $\bC$ and object $x$ in $\bD$,
		\item\stepcounter{generators}\label{product_presentation_2}
			$(a,x) \to (a,y)$ for each object $a$ in $\bC$ and generator $x \to y$ in $\bD$,
	\end{enumerate}
	and relations 

\begin{enumerate}[start=1,label={\bfseries R\arabic{relations}}]
\item\stepcounter{relations}  \label{product_presentation_3}
			$(a,x) \rightrightarrows (b,x)$ for each relation $a \rightrightarrows b$ in $\bC$ and object $x$ in $\bD$,
		\item\stepcounter{relations}\label{product_presentation_4}
			$(a,x) \rightrightarrows (a,y)$ for each object $a$ in $\bC$ and relation $x \rightrightarrows y$ in $\bD$, and 
		\item\label{stepcounter}\label{product_presentation_5} a {\bf swap} relation giving the commutativity of the square
		\[ \xymatrix{
			(a,x) \ar[d] \ar[r] & (b,x) \ar[d] \\
			(a,y) \ar[r] & (b,y)
		} \]
		for each each generator $a \to b$ in $\bC$ and generator $x \to y$ in $\bD$.
	\end{enumerate}
\end{lem}

\begin{lem}\label{slice_presentation}
	For each object $c$ in $\bC$, the slice category $(c \downarrow \bC)$ has a presentation with generators 
	\begin{enumerate}[start=1,label={\bfseries G\arabic{generators}}]
\item\stepcounter{generators}   $(c \to a) \to (c \to b)$ for each object  $c \to a$ in  $(c \downarrow \bC)$ and generator $a \to b$ in $\bC$
	\end{enumerate}
and relations
		\begin{enumerate}[start=1,label={\bfseries R\arabic{relations}}]
\item\stepcounter{relations}
$(c \to a) \rightrightarrows (c \to b)$ for each object $c \to a$ in  $(c \downarrow \bC)$ and relation $a \rightrightarrows b$ in $\bC$.
	\end{enumerate}
\end{lem}

\begin{rmk}
	The above two lemmas remain true if the presentations of $\bC$ and $\bD$ contain invertible generators. Of course, in that case we must take each of the corresponding generators in $\bC \times \bD$ and $(c \downarrow \bC)$ to be invertible as well. 
\end{rmk}

\begin{example}\label{present_e_cyclic_2}
Let $\sB C_n$ be as in \cref{present_e_cyclic} and let $\sE C_n \coloneqq (* \downarrow \sB C_n)$. Then \cref{slice_presentation,present_e_cyclic} 
 give a presentation for 
$\sE C_n$ with generators 
	\begin{enumerate}[start=1,label={\bfseries G\arabic{generators}}]
\item\stepcounter{generators}  
pairs  $(\sigma, a_k)$
where $\sigma$ is an element of $C_n$ and $a_k$ is a generator from \ref{item:cyclic_generator} 
\end{enumerate}
and relations
		\begin{enumerate}[start=1,label={\bfseries R\arabic{relations}}]
\item\stepcounter{relations}
 pairs $(\sigma, R)$
where $\sigma$ is an element of $C_n$ and $R$ is one of the relations in \ref{item:cyclic_relation}.
\end{enumerate}
In particular,  $\sE C_n$ has an invertible generator for every pair of objects and a relation for every triple of objects.
\end{example}

\begin{example}\label{present_e_symmetric}
Similarly, if $\sB\Sigma_n$ is as in \cref{presentation:symmetric_group} and $\sE\Sigma_n \coloneqq (* \downarrow \sB\Sigma_n)$, then \cref{slice_presentation,presentation:symmetric_group}
give a presentation for 
$\sE\Sigma_n$ with generators 
	\begin{enumerate}[start=1,label={\bfseries G\arabic{generators}}]
\item\stepcounter{generators}\label{present_e_symmetric_1}   pairs $(\sigma, \tau)$
where $\sigma$ is an element of $\Sigma_n$ and $\tau$ is a transposition
\end{enumerate}
and relations 
		\begin{enumerate}[start=1,label={\bfseries R\arabic{relations}}]
\item\stepcounter{relations}\label{present_e_symmetric_2} pairs $(\sigma, R)$
where $R$ is one of the relations \ref{squaretozero} to \ref{threecycle}.
\end{enumerate}
In the pair $(\sigma,\tau)$ we think of $\tau$ as inducing a map $\tau\colon \sigma \to \tau\sigma$.
In the pair $(\sigma,R)$ we think of $R$ as a  relation in $\Sigma_n$ between words of morphisms starting at $\sigma$.
\end{example}

Finally, suppose that $\bI$ and each of the abstract cliques $\underline{\doc}(i)$ have given presentations.
\begin{lem}\label{groth_presentation} The objects of  $\int_\bI \doc$ from \cref{clique_bijection} are the union over $i \in \ob\ \bI$ of the objects of $\underline{\doc}(i)$.  The morphisms are generated by 
\begin{enumerate}[start=1,label={\bfseries G\arabic{generators}}]
\item\stepcounter{generators}  
	generators for each of the categories ${\doc}(i)$ (vertical generators), and
	\item\stepcounter{generators}   an arrow $f\colon (i,x) \to (j,y)$ for each generator $f\colon i \to j$ in $\bI$ and  pair $(x,y)$ in some nonempty subset of ${\doc}(i) \times {\doc}(j)$ (horizontal generators).
\end{enumerate}
\end{lem}
There are nontrivial relations between these, but they are never needed for our proofs.

\begin{proof}
	Given any morphism $(i,x) \to (j,y)$, factor the map $i \to j$ into generators. For each generator pick a corresponding map $(i',x') \to (i'',x'')$ in the Grothendieck construction. Adding canonical isomorphisms in the fiber categories ${\doc}(i)$ these lifts can be composed.   The resulting composite is the original map $(i,x) \to (j,y)$ because it has the same image in $\bI$. By assumption, the added canonical isomorphisms can be written as composites of the vertical generators. Then the original morphism is a composite of vertical and horizontal generators.
\end{proof}

\section{Coherence for categories}\label{sec:coherence_cat}
In this section we prove the coherence theorems for bicategories, symmetric monoidal categories, and bicategories with shadow. Monoidal categories follow as a special case.

\subsection{Coherence for bicategories}\label{sec:coherence_bicat} Let $\oneBicat$ be the (1-)category whose objects are bicategories and whose morphisms are  strict functors of bicategories.  
In particular, the bicategories can have non-trivial associator and unitor isomorphisms, but they are strictly preserved by the functors.

Let $\Graphcat$ be the category whose objects are \emph{oriented} graphs and morphisms maps of graphs.  (A map of graphs takes vertices to vertices and edges to edges preserving adjacencies and orientations.)  
There is a forgetful functor
\[\fobic{-}\colon \oneBicat\to \Graphcat\]
that takes the underlying 0-cells and 1-cells, and forgets the 2-cells.  
This functor has a left adjoint $\frbic-$ that defines the free bicategory on a graph.  

\begin{present}\label{rmk:explicit_free_bicat} Let $\gLambda$ be a graph.    We will say that an ordered tuple of edges $X_1,\ldots , X_n$ in $\gLambda$  is {\bf composable} if the edges define a directed path in $\gLambda$.

The objects of $\frbic{\gLambda}$ are the  vertices of $\gLambda$. The 1-cells of $\frbic{\gLambda}$ 
 are parenthesizations of
\begin{equation}\label{eq:explicit_free_bicat_1_cell}
\underbrace{I_{A_0}\odot \cdots\odot  I_{A_0}}_{j_0}\odot X_1\odot \underbrace{I_{A_1}\odot \cdots\odot I_{A_1}}_{j_1}\odot X_2\odot  \ldots\odot  X_{n-1}\underbrace{I_{A_{n-1}}\odot \cdots \odot I_{A_{n-1}}}_{j_{n-1}}\odot X_n\odot \underbrace{I_{A_n}\odot \ldots \odot  I_{A_n}}_{j_n}
\end{equation}
for composable edges  $X_1,\ldots ,X_n$ of $\gLambda$. We usually write the units as $I$ without subscript, since the subscript is determined by its position in the expression:
\[ (I \odot (X_1 \odot X_2)) \odot (I \odot X_3) := (I_{A_0} \odot (X_1 \odot X_2)) \odot (I_{A_2} \odot X_3) \]
Equivalently, the 1-cells of $\frbic{\gLambda}$ are binary trees with leaves labeled by edges of $\gLambda$ or by formal units $I$, written in an order that makes them composable.

The 2-cells of  $\frbic{\gLambda}$ are generated by 
	\begin{enumerate}[start=1,label={\bfseries G\arabic{generators}}]
\item\stepcounter{generators}  \label{rmk:explicit_free_bicat_1} formal associator isomorphisms $\alpha\colon A \odot (B \odot C) \cong (A \odot B) \odot C$ and 
	\item\stepcounter{generators}\label{rmk:explicit_free_bicat_2} formal unitor isomorphisms $\ell\colon I \odot A \cong A$ and $r\colon A \odot I \cong A$. 
\end{enumerate}
Here $A$, $B$, and $C$ are any groups of parenthesized terms inside the larger word \eqref{eq:explicit_free_bicat_1_cell}. (In other words, we take all expanded instances of the associator and unitor maps.) Implicit in the word ``isomorphism'' is that we are taking these as invertible generators. So each one actually consists of two generators pointing in opposite directions, plus relations making them into inverses. In addition to these, the relations for $\frbic{\gLambda}$ are
	\begin{enumerate}[start=1,label={\bfseries R\arabic{relations}}]
\item\stepcounter{relations}\label{rmk:explicit_free_bicat_4} the pentagon relation for a bicategory,
\begin{equation}\label{eq:explicit_free_bicat_4}
	\xymatrix{X_1\odot (X_2\odot (X_3\odot X_4))\ar[r]^-\alpha\ar[d]^{1\odot \alpha}
		&(X_1\odot X_2)\odot (X_3\odot X_4)\ar[r]^-\alpha
		&((X_1\odot X_2)\odot X_3)\odot X_4
		\\
		X_1\odot ((X_2\odot X_3)\odot X_4)\ar[rr]^-\alpha
		&&(X_1\odot (X_2\odot X_3))\odot X_4\ar[u]_-{\alpha\odot 1}}
\end{equation}
\item\stepcounter{relations}\label{rmk:explicit_free_bicat_4a}  the triangle relation for a bicategory,
\begin{equation}\label{eq:explicit_free_bicat_4a}
	\xymatrix @R=1.5em{X_1\odot (I\odot X_2)\ar[rr]^-\alpha\ar[dr]_{1\odot l}&&(X_1\odot I)\odot X_2\ar[dl]^-{r\odot 1}
			\\&X_I\odot X_2}
\end{equation}
	\item\stepcounter{relations}\label{rmk:explicit_free_bicat_5} ({\bf whiskering}) 
		any two isomorphisms applied to disjoint regions in \eqref{eq:explicit_free_bicat_1_cell}  commute, and
	\item\stepcounter{relations}\label{rmk:explicit_free_bicat_6} ({\bf naturality}) 
		isomorphisms commute with any other isomorphism applied to the interior of one of its terms.
\end{enumerate}
\ref{rmk:explicit_free_bicat_5} guarantees that $\odot$ is a bifunctor and \ref{rmk:explicit_free_bicat_6} guarantees that the maps $\alpha$, $l$, and $r$ are natural isomorphisms, not just maps.
\end{present}

A {\bf formal diagram} of 2-cells in a bicategory $\sC$ is any diagram that lifts along the counit morphism $\frbic{\fobic{\sC}} \to \sC$. 

\begin{thm}[Coherence for bicategories]
\label{bicat_coherence_formal}
Every formal diagram of 2-cells in a bicategory  $\sC$ commutes.
\end{thm}
We prove this in stages. The first step is to handle the associators.

For an $n$-tuple of composable edges $X_1, \ldots ,X_n$ in a graph $\gLambda$ let $A_\gLambda (X_1,\ldots , X_n)$ be the subcategory of $\frbic{\gLambda}$ whose objects are all parenthesization of the expression
\begin{equation}\label{eq:lots_of_xs} X_1\odot X_2\odot\ldots\odot X_n
\end{equation}
and morphisms generated by the associator isomorphisms.  
(Since graphs don't have identity edges, none of the $X_i$ are unit 1-cells.)
We define $A_\gLambda (X_1,\ldots , X_n)$ as a subcategory of $\frbic{\gLambda}$, not as the category generated by associator maps with only the pentagon axiom relation between them. A priori, there could be more relations coming from composites of morphisms that pass outside of $A_\gLambda (X_1,\ldots , X_n)$.

\begin{lem}\label{lem:bicat_assoc_induction}
$A_\gLambda (X_1,\ldots , X_n) \subseteq \frbic{\gLambda}$ is a clique in $\frbic{\gLambda}$.
\end{lem}

\begin{proof}
	The proof is by induction on $n$. There is nothing to check when  $n = 0,1,2$.

When $n\geq 3$, for each $1\leq i<n$ let $A_\gLambda ^i(X_1,\ldots , X_n) \subseteq A_\gLambda(X_1,\ldots , X_n)$ be a subcategory whose objects are the parenthesizations of the form
\[ (X_1 \odot \ldots \odot X_i) \odot (X_{i+1} \odot \ldots \odot X_n). \]
We define its morphisms to be those generated by associator maps in each of the two blocks $X_1 \odot \ldots \odot X_i$ and $X_{i+1} \odot \ldots \odot X_n$. In other words, it is a product category
\[ A_\gLambda ^i(X_1,\ldots , X_n) \cong A_\gLambda (X_1,\ldots , X_i) \times A_\gLambda (X_{i+1},\ldots,X_n). \]
By inductive hypothesis these factors are cliques, hence $A_\gLambda ^i(X_1,\ldots , X_n)$ is a clique.

Alternatively, if we assign each parenthesization to the index of the term to the left of its outermost composition, then $A_\gLambda ^i(X_1,\ldots , X_n)$ is those parenthesizations of index $i$. For example, $(X_1\odot X_2)\odot X_3$ has index 2 and belongs to $A_\gLambda ^2(X_1,X_2,X_3)$, while $X_1\odot (X_2\odot X_3)$ has index 1 and belongs to $A_\gLambda ^1(X_1,X_2,X_3)$. See \cref{fig:bicat_assoc_induction} for a picture of  $A_{\gLambda}(X_1,X_2,X_3,X_4,X_5)$.

\usetikzlibrary{fit}

\begin{figure}
\begin{tikzpicture}[
    mystyle/.style={%
      label={right:\pgfkeysvalueof{/pgf/minimum width}},
    },
   my style/.style={%
     append after command={
       \pgfextra{\node [right] at (\tikzlastnode.mid east) {\tikzlastnode};}
     },
   },
  ]
 \node[draw
	] (N) at (1.3*4, 7){$\bullet(\bullet(\bullet(\bullet\bullet)))$
};
 \node[draw
	] (M) at (1.3*4, 5){$(\bullet\bullet)(\bullet(\bullet\bullet))$
};
 \node[draw
	] (L) at (1.3*2, 7){$\bullet(\bullet((\bullet\bullet)\bullet))$
};
 \node[draw
	] (K) at (1.3*0, 7){$\bullet((\bullet(\bullet\bullet))\bullet)$
};
 \node[draw
	] (J) at (1.3*6, 8){$\bullet((\bullet\bullet)(\bullet\bullet))$
};
 \node[draw
	] (I) at (1.3*2, 5){$(\bullet\bullet)((\bullet\bullet)\bullet)$
};
 \node[draw
	] (H) at (1.3*-2, 8){$\bullet(((\bullet\bullet)\bullet)\bullet)$
};
 \node[draw
] (G) at (1.3*0, 1){$(\bullet(\bullet(\bullet\bullet)))\bullet$
};
\node[draw
] (F) at (1.3*6, 3){$(\bullet(\bullet\bullet))(\bullet\bullet)$
};
\node[draw
] (E) at (1.3*2, 1){$((\bullet\bullet)(\bullet\bullet))\bullet$
};
\node[draw
] (D) at (1.3*-2, 0){$(\bullet((\bullet\bullet)\bullet))\bullet$
};
\node[draw
] (C) at (1.3*6, 0){$((\bullet(\bullet\bullet))\bullet)\bullet$
};
\node[draw
] (B) at (1.3*4, 3){$((\bullet\bullet)\bullet)(\bullet\bullet)$
};
\node[draw
] (A) at (1.3*4, 1){$((((\bullet\bullet)\bullet)\bullet)\bullet$
};
\draw[-](M)--(N);
\draw[-](L)--(N);
\draw[-](K)--(L);
\draw[-](J)--(N);
\draw[-](I)--(L);
\draw[-](I)--(M);
\draw[-](H)--(K);
\draw[-](H)--(J);
\draw[-](G)--(K);
\draw[-](F)--(J);
\draw[-](E)--(G);
\draw[-](E)--(I);
\draw[-](D)--(H);
\draw[-](D)--(G);
\draw[-](C)--(D);
\draw[-](C)--(F);
\draw[-](B)--(M);
\draw[-](B)--(F);
\draw[-](A)--(B);
\draw[-](A)--(C);
\draw[-](A)--(E);
\end{tikzpicture}
\caption{The clique $A_{\gLambda}(X_1,X_2,X_3,X_4,X_5)$. Each horizontal layer (including diagonal maps) is a single sub-clique $A_{\gLambda}^i$.
 Edges are only added for single instances of the associator.}\label{fig:bicat_assoc_induction}
\end{figure}
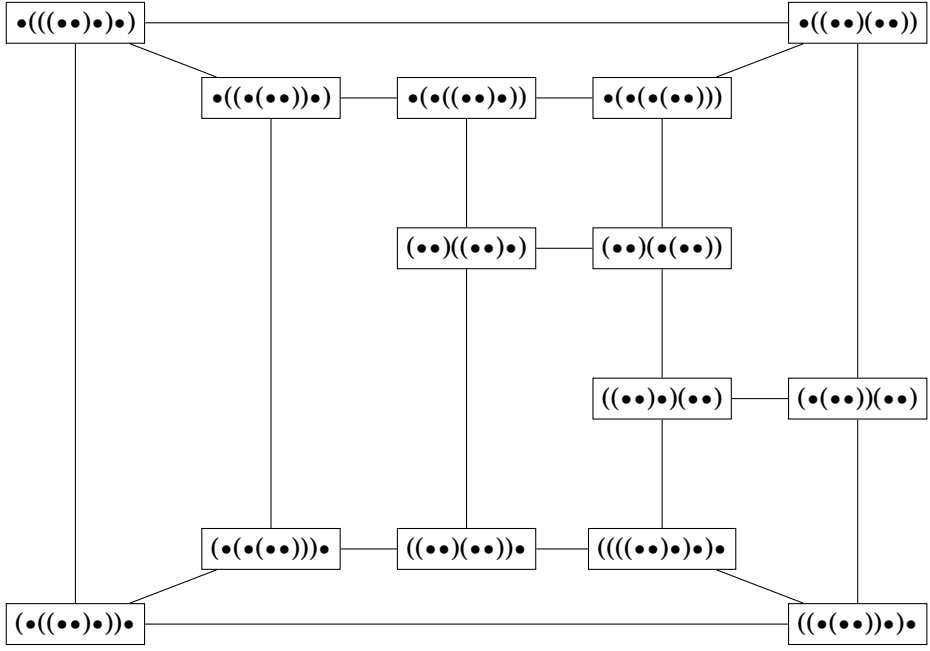

For each pair $i < j$ we define an (invertible) map of cliques
\begin{equation}\label{assoc_map_cliques_2}
A_\gLambda ^{i}(X_1,\ldots , X_n)\leftrightarrow A_\gLambda ^{j}(X_1,\ldots , X_n).
\end{equation}
by taking an admissible model (\cref{rmk:thick_maps_partially_defined}) for each object in the three-fold product
\[ A_\gLambda (X_1,\ldots , X_i) \times A_\gLambda (X_{i+1},\ldots,X_j) \times A_\gLambda (X_{j+1},\ldots,X_n). \]
These three choices of parenthesization define objects in the cliques $A_\gLambda ^{i}$ and $A_\gLambda ^{j}$ as indicated, and we map between them by the associator isomorphism.
\begin{align*}
		(X_1 \odot \ldots \odot X_i) &\odot ((X_{i+1} \odot \ldots \odot X_j) \odot (X_{j+1} \odot \ldots \odot X_n))
	\\
		 &\xleftrightarrow{\alpha} ((X_1 \odot \ldots \odot X_i) \odot (X_{i+1} \odot \ldots \odot X_j)) \odot (X_{j+1} \odot \ldots \odot X_n).
	\end{align*}
These associators commute with the canonical isomorphisms between different models by naturality \eqref{rmk:explicit_free_bicat_6}. We therefore have a well-defined map of cliques \eqref{assoc_map_cliques_2}.

For instance, in \cref{fig:bicat_assoc_induction}, the top horizontal layer (with five objects) forms a clique that maps to the bottom horizontal layer (with five objects), but there are only two admissible models for this map, the ones in the left two columns. Those are the two places where we can jump from one clique to the other by a single associator map.

Now we assemble the categories $A_\gLambda ^i(X_1,\ldots , X_n)$ into a diagram of cliques.  Let $\bI_{n-1}$ be the $(n-1)$-fold subdivided interval category
	\[ \xymatrix{ y_1 \ar@{<->}[r] & y_2 \ar@{<->}[r] & \cdots \ar@{<->}[r] & y_{n-1} } \] and choose the presentation of  $\bI_{n-1}$ with invertible generators $y_i \leftrightarrow y_j$ for each pair $i < j$.  The relations are given by 
\[(y_i \leftrightarrow y_j \leftrightarrow y_k)=(y_i \leftrightarrow y_k)\]
for $i<j<k$.

Define a diagram of cliques $\docf$ in $A_\gLambda (X_1,\ldots , X_n)$ with domain $\bI_{n-1}$  by taking $\doc(i)$ to be $A_\gLambda^i (X_1,\ldots , X_n)$.  The image of the generator  $y_i \leftrightarrow y_j$ is the map in \eqref{assoc_map_cliques_2}.  The condition imposed by the relation 
$(y_i \leftrightarrow y_j \leftrightarrow y_k)=(y_i \leftrightarrow y_k)$ can be checked on a single element. We therefore fix parenthesizations for each of the four blocks
\[ X_1 \odot \ldots \odot X_i, \quad X_{i+1} \odot \ldots \odot X_j, \quad X_{j+1} \odot \ldots \odot X_k, \quad X_{k+1} \odot \ldots \odot X_n \]
and check that the clique maps $A_\gLambda^i \to A_\gLambda^j$ and $A_\gLambda^j \to A_\gLambda^k$ on this model agree (along canonical isomorphisms in $A_\gLambda^i$ and $A_\gLambda^k$) with the clique map $A_\gLambda^i \to A_\gLambda^k$. This becomes exactly the pentagon axiom (\ref{rmk:explicit_free_bicat_4}). For example, the three tall pentagonal regions in \cref{fig:bicat_assoc_induction} all arise this way.

\cref{clique_bijection} defines a functor 
	\begin{equation}\label{functor:lem:bicat_assoc_induction} \int_{\bI_{n-1}} \doc \to A_\gLambda (X_1,\ldots , X_n). \end{equation}
	This functor is a bijection on objects. We check it is surjective on morphisms by writing out the generators from \cref{groth_presentation} and checking that together they hit all of the generators of $A_\gLambda (X_1,\ldots , X_n)$. In particular, every expanded instance of an associator map occurs as some morphism in the image of $\int_{\bI_{n-1}} \doc$. By 
\cref{lem:consequences_of_projection_2}, $\int_{\bI_{n-1}} \doc$ is a clique, so the functor in \eqref{functor:lem:bicat_assoc_induction}  must be faithful. Therefore it is an isomorphism of categories and $A_\gLambda (X_1,\ldots , X_n)$ is a clique. This finishes the induction.
\end{proof}

Recall that $\sI$ is the category of finite totally ordered sets and injective maps from \cref{inj_presentation}. Let $\sI[\sI^{-1}]$ be the localization where every map is made invertible. Note that $\sI[\sI^{-1}]$ has an initial object (the empty set), so it is an abstract clique.

Continue to fix a single $n$-tuple of composable edges $X_1, \ldots ,X_n$ in a graph $\gLambda$. Then for each tuple of non negative integers $j_0,j_1,\ldots j_n$ let 
\[U_\gLambda (X_1,\ldots , X_n,j_0,j_1,\ldots j_n)\]
denote the clique in $\frbic{\gLambda}$
\[A_\gLambda (\underbrace{I,I,\ldots I}_{j_0},X_1,\underbrace{I,I,\ldots I}_{j_1},X_2,\ldots, X_{n-1},\underbrace{I,I,\ldots I}_{j_{n-1}},X_n,\underbrace{I,I,\ldots I}_{j_n}). \] 

\begin{lem}[\cref{rmk:plan_of_proofs}\ref{rmk:plan_of_proofs_1}]\label{bicat_coherence_diagram_assoc_unit}
There is a diagram  of cliques in $\frbic{\gLambda}$
 indexed on $\prod^{n+1} \sI[\sI^{-1}]$ where the image of $(j_0,j_1,\ldots j_n)$ is the clique $U_\gLambda (X_1,\ldots , X_n,j_0,j_1,\ldots j_n)$.
\end{lem}

\begin{proof}
The generators and relations of  $\prod^{n+1} \sI[\sI^{-1}]$ are given by \cref{product_presentation} and \cref{inj_presentation}.  Each generator is a coface map in one of the factors, say the $i$th coface map in the $m$th factor. We assign it to the clique map
\[ U_\gLambda (X_1,\ldots , X_n,j_0,\ldots j_m, \ldots j_n)
\to U_\gLambda (X_1,\ldots , X_n,j_0,\ldots j_m + 1, \ldots j_n) \]
which inserts a unit between the $(i-1)$st and $i$th terms in the block of $j_m$ copies of $I$. There are two unit maps that we could use to make this insertion, $\ell$ and $r$ from \ref{rmk:explicit_free_bicat_2}, but the triangle axiom \ref{rmk:explicit_free_bicat_4a} implies these two possibilities agree after composing with the associator.  This map is compatible with the canonical isomorphisms by whiskering (\ref{rmk:explicit_free_bicat_5}).  

For the relations, the swap relation (\ref{product_presentation_5})  
follows from whiskering (\ref{rmk:explicit_free_bicat_5}).
	The relations  \ref{product_presentation_3} and \ref{product_presentation_4} become the relation \ref{faceface} within each copy of $\sI$, namely : $d^i d^j = d^{j+1} d^i$ whenever $i \leq j$. When $i < j$, this relation holds by whiskering (\ref{rmk:explicit_free_bicat_5}). When $i = j$, it is the commutativity of the following diagram for words $W$ in $X_1,\ldots, X_n$ and copies of $I$.
	\[ \xymatrix{
		(I \odot I) \odot W \ar@{<->}[dr]^-{l \odot 1} \ar@{=}[r] & (I \odot I) \odot W \ar@{<->}[d]^-{r \odot 1} \ar@{<->}[r]^-\alpha & I \odot (I \odot W) \ar@{<->}[d]^-{1 \odot l} \\
		& I \odot W \ar@{<->}[dr]_-l \ar@{=}[r] & I \odot W \ar@{<->}[d]^-l \\
		& & W
	} \]
	The bottom region  commutes by definition and the square region commutes by  \ref{rmk:explicit_free_bicat_4a}. The left triangle is the assumption that $l = r$ when applied to a unit 1-cell $I$. One either adds this to the list of bicategory axioms, or deduces it from the pentagon and triangle axioms using the classic argument of Kelly \cite[Thms 6 and 7]{kelly_simplified_coherence}. 
\end{proof}

\begin{defn}
If $c$ is an object of a category $\bC$, let $\comp{\bC}{c}$ be the component of $\bC$ that contains $c$.
\end{defn}
As an example, if $X_i\in \sB(A_{i-1},A_i)$ and $W$ is a particular model for the product $X_1 \odot \ldots \odot X_n$, then the clique  $\bigodot_{i=1}^nX_i$ from \cref{ex:define_odot_clique} is the component $\comp{\frbic{\gLambda}(A_0,A_n)}{W}$.

Let $\docf$ be the diagram of cliques in \cref{bicat_coherence_diagram_assoc_unit}.  \cref{clique_bijection} 
defines a functor 
\begin{equation}\label{bicat_coherence_functor_assoc_unit_1}
\int_{\prod^{n+1}\sI[\sI^{-1}]} \doc\to \frbic{\gLambda}(A_0,A_n).
\end{equation}
Since $\int_{\prod^{n+1}\sI[\sI^{-1}]} \doc$ is connected,
\eqref{bicat_coherence_functor_assoc_unit_1} defines a functor 
\begin{equation}\label{bicat_coherence_functor_assoc_unit_2}
\int_{\prod^{n+1}\sI[\sI^{-1}]} \doc\to \comp{\frbic{\gLambda}(A_0,A_n)}{W}
\end{equation}
for any model $W$  from $A_\gLambda (X_1,\ldots , X_n)$. (\cref{rmk:plan_of_proofs}\ref{rmk:plan_of_proofs_2})

\begin{thm}[\cref{rmk:plan_of_proofs}\ref{rmk:plan_of_proofs_3}]\label{bicat_coherence} The functor in \eqref{bicat_coherence_functor_assoc_unit_2}
is an isomorphism of categories.
\end{thm}

\begin{proof}
By definition it is an isomorphism on objects. Since the diagram of cliques  in \cref{bicat_coherence_diagram_assoc_unit} 
includes all possible instance of the maps $\alpha$, $l$, and $r$, the functor \eqref{bicat_coherence_functor_assoc_unit_2} 
 is also surjective on morphisms.

Since $\prod^{n+1} \sI[\sI^{-1}]$ is an abstract clique (it has $\emptyset$ as an initial object), 
\cref{lem:consequences_of_projection_2} and \cref{bicat_coherence_diagram_assoc_unit} imply that
\begin{equation}\label{bicat_coherence_functor_assoc_unit_3}
\int_{\prod^{n+1} \sI[\sI^{-1}]} \doc
\end{equation} is an abstract clique. Therefore the functor \eqref{bicat_coherence_functor_assoc_unit_2}  is faithful, and so it is an isomorphism of categories.   
\end{proof}

\begin{cor}[\cref{rmk:plan_of_proofs}\ref{rmk:plan_of_proofs_4}]
Each component $\comp{\frbic{\gLambda}(A_0,A_n)}{W}$ is a clique. Equivalently, $\frbic{\gLambda}(A_0,A_n)$ is a thin groupoid.
\end{cor}

This finishes the proof of coherence for bicategories (\cref{bicat_coherence_formal}), since formal diagrams of 2-cells in $\sC$ are the image of diagrams in $\frbic{\fobic{\sC}}$, and this establishes that all diagrams of 2-cells in $\frbic{\fobic{\sC}}$ commute.

We also have the following consequence of \cref{bicat_coherence} that we will use when proving the coherence theorems for lax functors (\cref{lax_not_coherence}).  Recall the clique defined in \cref{ex:define_odot_clique_F}.  
\begin{cor}\label{built_diagram_of_cliques:lem_coherence_bicat} 
Let $F\colon \sB\to \sB'$ be a lax functor of bicategories. 
Each coface map $d^i\colon \underline{k}\to \underline{k+1}$ (\ref{gen:coface}) defines a map of cliques 
\[ \bigodot_{j \in \underline{k}} F\left( \bigodot_{i \in \alpha^{-1}(j)} X_i \right) \to
\bigodot_{j \in \underline{k+1}} F\left( \bigodot_{i \in \alpha^{-1}(j)} X_i \right). \]
Each coboundary map $s^i\colon \underline{k}\to \underline{k-1}$ (\ref{gen:codegeneracy})  defines a map of cliques 
\[ \bigodot_{j \in \underline{k}} F\left( \bigodot_{i \in \alpha^{-1}(j)} X_i \right) \to
\bigodot_{j \in \underline{k-1}} F\left( \bigodot_{i \in \alpha^{-1}(j)} X_i \right). \]
\end{cor}
The intuition is that coface maps add new points to the codomain and  are sent to 
\[i\colon I_{F(A)} \to F(I_{A}).\] 
Codegeneracy maps that  fold points together are sent to the map
\[m\colon F(X) \odot F(X') \to F(X \odot X').\]

\begin{proof}
Let $W_j$ be a model of $\bigodot_{i \in \alpha^{-1}(j)} X_i$.

For the coface map $d^i$ first consider models that have exactly one unit $I_{A_i}$ in between $W_i$ and $W_{i+1}$. To each model we apply the unit morphism $i$ to the unique $I_{A_i}$. \cref{bicat_coherence} implies that we could broaden our class of admissible models to those with at least one unit $I_{A_i}$ between $W_i$ and $W_{i+1}$ and take our map to be one that applies $i$ to any unit object.  

For each codegeneracy map $s^i$ 
the admissible models are those for which the $k$-fold tensor product places a single tensor between $F(W_i)$ and $F(W_{i+1})$, and no other units or parentheses.  (The model contains the term $F(W_i) \odot F(W_{i+1})$.) To these models we apply the composition morphism $m$. 
By \cref{bicat_coherence}, the canonical isomorphism between any two admissible models can be chosen to be one that does not change $F(W_i) \otimes F(W_{i+1})$.  Then \ref{rmk:explicit_free_bicat_6} demonstrates that the maps on the two models are compatible and gives a well-defined map of cliques.
\end{proof}

\subsection{Symmetric monoidal categories}
In this section we use the cliques in bicategories (and hence monoidal categories) constructed in \cref{sec:coherence_bicat} 
to construct cliques in symmetric monoidal categories.

Let $\symcat$ be the category whose objects are symmetric monoidal categories and morphisms are strict symmetric monoidal functors.  There is a forgetful functor 
\[\fosmc{-}\colon \symcat\to \setcat\]
to the category of sets that takes the set of objects and forgets the morphisms and symmetric monoidal structure. 
Let $\frsmc{-}$ denote the left adjoint of $\fosmc{-}$. 
\begin{present}\label{rmk:explicit_free_symcat}
For a set $\sLambda$, objects of $\frsmc{\sLambda}$ are parenthesiziations of 
\begin{equation}\label{expand_rep_symmetric}
\underbrace{I\otimes I\otimes \cdots\otimes I}_{j_0}\otimes \slambda_1\otimes \underbrace{I\otimes I\otimes \cdots\otimes I}_{j_1}\otimes \slambda_2\otimes \ldots\otimes  \underbrace{I\otimes I\otimes \cdots \otimes I}_{j_{n-1}}\otimes \slambda_n\otimes \underbrace{I\otimes I\otimes \ldots \otimes I}_{j_n}
\end{equation}
for elements  $\slambda_1,\ldots ,\slambda_n$ of $\sLambda$.
Generators for the morphisms in $\frsmc{\sLambda}$ are \ref{rmk:explicit_free_bicat_1}, \ref{rmk:explicit_free_bicat_2} and 
	\begin{enumerate}[start=1,label={\bfseries G\arabic{generators}}]
\item\stepcounter{generators}  \label{expand_rep_symmetric_generator}
expanded instances of the symmetry isomorphisms $\gamma\colon A \otimes B \cong B \otimes A$.
\end{enumerate}
The relations for $\frsmc{\sLambda}$ are \ref{rmk:explicit_free_bicat_4} to \ref{rmk:explicit_free_bicat_6},
		\begin{enumerate}[start=1,label={\bfseries R\arabic{relations}}]
\item\stepcounter{relations}\label{expand_rep_symmetric_relation_1}
$\gamma^2 = 1$
\item\stepcounter{relations}\label{expand_rep_symmetric_relation_2} the triangle relating the symmetry and the unit maps 
\[\xymatrix @R=1.5em{X\otimes I\ar[rd]_r\ar[rr]^-\gamma&&I\otimes X\ar[dl]^l\\
&X}\]
\item\stepcounter{relations}\label{expand_rep_symmetric_relation_3}the hexagon relating the symmetry and associativity
\[\xymatrix{(X\otimes Y)\otimes Z\ar[r]^{\gamma\otimes \id}\ar[d]^\alpha
	&(Y\otimes X)\otimes Z\ar[r]^-\alpha
	&Y\otimes (X\otimes Z)\ar[d]^{\id\otimes \gamma}
	\\
	X\otimes (Y\otimes Z)\ar[r]^-\gamma&(Y\otimes Z)\otimes X\ar[r]^-\alpha&Y\otimes (Z\otimes X)
}\]
\item\stepcounter{relations} \label{expand_rep_symmetric_relation_4} naturality relations that $\gamma$ commutes with morphisms applied to the two smaller words (making $\gamma$ a natural transformation).
\end{enumerate} 
\end{present}

Ignoring the unit elements, a morphism $\phi$ in $\frsmc{\sLambda}$ with domain a parenthesization of  \eqref{expand_rep_symmetric} induces permutation of the $\slambda_i$.  This induces  a permutation $P(\phi)$ of $\{1,\ldots, n\}$.

\begin{lem}\label{lem:underlying_functor_symmetric}
For each component of $\frsmc{\sLambda}$, the assignment $\phi\mapsto P(\phi)$ defines an {\bf {\punderlying} permutation} functor
\begin{equation*}
 P\colon \comp{\frsmc{\sLambda}}{W} \to 
\sB \Sigma_n.
\end{equation*}
\end{lem}

As before, a {\bf formal diagram} of morphisms in a symmetric monoidal category $\bC$ is a diagram that lifts against the counit 
\[\frsmc{\fosmc{\bC}}\to \bC.\] 

\begin{defn}
A formal diagram of morphisms in a symmetric monoidal category is {\bf \blacktie} if, for every pair of parallel morphisms, the {\punderlying} permutations of the two composites agree. Equivalently, the diagram commutes after applying the functor from \cref{lem:underlying_functor_symmetric}. (When the $\slambda_i$ are distinct, every formal diagram is {\blacktie}.)
\end{defn}

\begin{thm}[Coherence for symmetric monoidal categories]\label{symcat_coherence_formal}
Every {\blacktie} diagram of morphisms in a symmetric monoidal category commutes.
\end{thm}

We proceed immediately into the proof. Fix an $n$-tuple of objects $\slambda_1,\ldots ,\slambda_n$ in a set $\sLambda$. Since a monoidal category is a bicategory with a single 0-cell, \cref{bicat_coherence} supplies a clique $\bigotimes_{i=1}^n\slambda_i$ in $\frsmc{\sLambda}$ consisting of associator and unitor maps, but where the ordering of the $\slambda_i$ is never altered.

\begin{lem}[\cref{rmk:plan_of_proofs}\ref{rmk:plan_of_proofs_1}]\label{symcat_coherence_diagram_symmetry}
There is a diagram of cliques 
indexed by $\sE\Sigma_n$ (\cref{present_e_symmetric}) where the image of $\sigma \in \Sigma_n$ is the clique $\bigotimes_{i=1}^n\slambda_{\sigma(i)}$.
\end{lem}

\begin{proof}
Each transposition (\ref{present_e_symmetric_1}) is sent to an instance of $\gamma$ (\ref{expand_rep_symmetric_generator}) and is well-defined by whiskering (\ref{rmk:explicit_free_bicat_5}). This respects the $\tau_i\tau_j$ relations (\ref{farawaycommute}) by whiskering (\ref{rmk:explicit_free_bicat_5}), the $\tau_i^2$ relations (\ref{squaretozero}) by the relation in $\frsmc{ \sLambda}$ that $\gamma^2 = 1$ (\ref{expand_rep_symmetric_relation_1}), and the $(\tau_i\tau_{i+1})^3$ relations (\ref{threecycle}) by the commutativity of the diagram
\[ \xymatrix @R=1.9em{
	 (Y \otimes X) \otimes Z   \ar@{<->}[d]^-{\gamma \otimes 1}\ar[rr]^-\alpha
	&& Y \otimes (X \otimes Z) \ar@{<->}[d]^-{1 \otimes \gamma}   
	\\
	(X\otimes Y)\otimes Z\ar@{<->}[d]^\alpha&&Y\otimes (Z\otimes X)\ar@{<->}[d]^\alpha
	\\
	X \otimes (Y \otimes Z) \ar@{<->}[d]^-{1 \otimes \gamma} \ar@{<->}[rr]^-{\gamma} &&
	(Y \otimes Z) \otimes X \ar@{<->}[d]^-{\gamma \otimes 1} \\
	X \otimes (Z \otimes Y)  \ar@{<->}[rr]^-{\gamma} \ar[d]^\alpha&&
	(Z \otimes Y) \otimes X \ar[d]^\alpha
	\\
	X \otimes (Z \otimes Y) &&Z\otimes (Y\otimes X)
	\\
	 Z \otimes X \otimes Y \ar@{<->}[u]_-{\gamma \otimes 1}\ar[rr]^-\alpha
	&&Z \otimes X \otimes Y \ar@{<->}[u]_-{1 \otimes \gamma}.
} \]
The hexagons commute by \ref{expand_rep_symmetric_relation_3} and the rectangle commutes by naturality of $\gamma$ (\ref{expand_rep_symmetric_relation_4}).
\end{proof}

Let $\docf$ be the diagram of cliques in \cref{symcat_coherence_diagram_symmetry}.
 \cref{clique_bijection} defines a functor 
into one component of the free symmetric monoidal category (\cref{rmk:plan_of_proofs}\ref{rmk:plan_of_proofs_2})
\begin{equation}\label{symmoncat_functor_symmetry_2}
\int_{\sE\Sigma_n} \doc
\to \comp{\frsmc{\sLambda}}{W} 
\end{equation}

\begin{thm}[\cref{rmk:plan_of_proofs}\ref{rmk:plan_of_proofs_3}]\label{thm:sym_coherence_iso}
If all the $\slambda_i$ are distinct then \eqref{symmoncat_functor_symmetry_2} is an isomorphism of categories. 
\end{thm}

\begin{proof}
\cref{lem:consequences_of_projection_2} implies 
$\int_{\sE\Sigma_n} \doc
$ is an abstract clique. 

If the objects $X_1,\ldots, X_n$ are distinct, the functor \eqref{symmoncat_functor_symmetry_2} 
 is a bijection on objects and surjective on morphisms (since each instance of $\alpha$, $l$, $r$, and $\gamma$ has a preimage by construction). It is automatically faithful since the source is thin. Therefore it is an isomorphism of categories.
\end{proof}

\begin{cor}[\cref{rmk:plan_of_proofs}\ref{rmk:plan_of_proofs_4}]\label{cor:sym_coherence_iso}
Let $W$ be a model for $\bigotimes_{i=1}^n\slambda_{i}$.
 If the $\slambda_i$ are distinct then
$\comp{\frsmc{\sLambda}}{W}$  is thin. 
More generally, the underlying permutation functor induces an equivalence of categories
\[\comp{\frsmc{\sLambda}}{W} \to \prod_i \sB\Sigma_{k_i} \]
for a certain subgroup $\prod_i \Sigma_{k_i} \leq \Sigma_n$ of block permutations.
\end{cor}

\begin{proof}
If the elements $\slambda_1,\ldots ,\slambda_n$ are not distinct, then \eqref{symmoncat_functor_symmetry_2} still defines a clique, but it is no longer an isomorphism of categories. Rather, it induces an isomorphism out of the quotient of the clique by the free action of the group $\prod_i \Sigma_{k_i}$, acting by permuting the repeats of each distinct element. 
Therefore $\comp{\frsmc{\gLambda}}{W}$ is equivalent to the category $\prod_i \sB\Sigma_{k_i}$, and the equivalence sends each morphism to its underlying permutation.
\end{proof}

This finishes the proof of coherence for symmetric monoidal categories (\cref{symcat_coherence_formal}).

\begin{example}\label{building_block_clique_symmetric_monoidal_functor}
Let $F\colon \bC\to \bD$ be a lax monoidal functor from a symmetric monoidal category $\bC$ to a monoidal category $\bD$, and let $X_1$, $\ldots$, $X_n$ denote \emph{distinct} objects in $\bC$. For each map of finite sets $\alpha\colon \underline{n} \to \underline{k}$ (not necessarily preserving the ordering!),  the clique from 
\cref{ex:define_odot_clique_F} can be extended to a clique we denote with the same notation
\[ \bigotimes_{j \in [k]} F\left( \bigotimes_{i \in \alpha^{-1}(j)} X_i \right) .\]
It has an object for each ordering of each of the preimages $\alpha^{-1}(j)$ and each model for their tensor product, and the tensor product on the outside. We include the symmetry isomorphisms $\gamma$, but \emph{only} inside the copies of $F$, not on the outside. In other words, we are taking a product of $k$ different cliques from the free symmetric monoidal category $\frsmc{ \sLambda}$ and one clique from the free monoidal category.

\end{example}

\subsection{Shadowed bicategories}\label{sec:shadowed_bicat}
Let $\oneshBicat$ be the category whose objects are bicategories with shadow $(\sC,\sC_{\mathrm{Sh}})$ (see e.g. \cite{p:thesis,ps:bicat}) and whose morphisms are strict maps (see \cref{sec:lax_shadow}).
There is a forgetful functor 
\begin{equation}\label{eq:shad_bicat_forget}
 \oneshBicat\to \Graphcat
\end{equation}
whose value on $(\sC,\sC_{\mathrm{Sh}})$ is $\fobic{\sC}$ (from \cref{sec:coherence_bicat}). Let $(\frbic{-}, \frsbit{-})$ be the left adjoint of the functor in \eqref{eq:shad_bicat_forget}.  
The underlying bicategory $\frbic{-}$ is as in \cref{rmk:explicit_free_bicat} and the shadow category $\frsbit{-}$ has the following presentation.

\begin{present}
For a graph $\gLambda$, let $\frsbit{\gLambda}$ be the category with objects the  set of {\emph endomorphism} 1-cells of $\frbic{\gLambda}$, with a $\sh{-}$ written around them. So for example
\[ \sh{ \ ((X_1 \odot I) \odot X_2) \odot (X_3 \odot X_4) \ }. \]

Generators for the morphisms of $\frsbit{\gLambda}$  are \ref{rmk:explicit_free_bicat_1}, \ref{rmk:explicit_free_bicat_2}, and 
	\begin{enumerate}[start=1,label={\bfseries G\arabic{generators}}]
\item\stepcounter{generators} \label{gen:rotator} 
rotator maps $\theta\colon \sh{A \odot B} \cong \sh{B \odot A}$. 
	\end{enumerate}
 Since $\theta$  can only be applied to the outermost tensor product in a given word, there are no expanded instances of $\theta$.

The relations for morphisms of $\frsbit{\gLambda}$  are  \ref{rmk:explicit_free_bicat_4} to \ref{rmk:explicit_free_bicat_6},
		\begin{enumerate}[start=1,label={\bfseries R\arabic{relations}}]

		\item\stepcounter{relations} \label{fig:shadow_unit_coherence} the diagram relating $\theta$ and the unit isomorphisms 
\[\xymatrix @R=1.5em{
	\sh{X \odot I} \ar@{<->}[dr]_-r \ar@{<->}[rr]^-\theta && \sh{I \odot X} \ar@{<->}[dl]^-l \\
	& \sh{X}. &
}\]
\item\stepcounter{relations} \label{fig:shadow_associativity_coherence} the diagram relating $\theta$ and the associators	\[\xymatrix{
		\sh{(X \odot Y) \odot Z} \ar@{<->}[d]^-\alpha \ar@{<->}[r]^-\theta &
		\sh{Z \odot (X \odot Y)} \ar@{<->}[r]^-\alpha &
		\sh{(Z \odot X) \odot Y} \ar@{<->}[d]^-\theta
		\\
		\sh{X \odot (Y \odot Z)} \ar@{<->}[r]^-\theta &
		\sh{(Y \odot Z) \odot X} \ar@{<->}[r]^-\alpha &
		\sh{Y \odot (Z \odot X)},
	} \]
	and
\item\stepcounter{relations}\label{theta:natural} naturality relations that $\theta$ commutes with morphisms applied to the two smaller words (making $\theta$ a natural transformation).
	\end{enumerate}

The shadow functor for the pair $(\frbic{\gLambda}, \frsbit{\gLambda})$
applies $\sh{-}$ to regard 1-cells in $\frbic{\gLambda}$ as objects of $\frsbit{\gLambda}$.
\end{present}

A morphism $\phi$ in $\frsbit{\gLambda}$  with domain a parenthesization of 
\[ \sh{ \ \underbrace{I\odot I\odot \cdots \odot  I}_{j_0}\odot X_1\odot\underbrace{I\odot I\odot\cdots \odot I}_{j_1}\odot X_2\odot\ldots\odot \underbrace{I\odot I\odot\cdots \odot I}_{j_{n-1}}\odot X_n\odot \underbrace{I\odot I\odot \cdots \odot I}_{j_n}
\ } \]
induces a cyclic permutation $P(\phi)$  of the set $\{1,\ldots, n\}$.
\begin{lem}\label{lem:underlying_functor_shadow}
For each component of $\frsbit{\gLambda}$, the assignment $\phi\mapsto P(\phi)$ defines a {\bf {\punderlying} cyclic permutation} functor 
\begin{equation*}
 P\colon \comp{\frsbit{\gLambda}}{W} \to 
\sB C_n.
\end{equation*}
\end{lem}

If $(\sB,\sB_{\mathrm{Sh}})$ is a shadowed bicategory, a {\bf formal diagram} in $\sB_{\mathrm{Sh}}$ is any diagram that lifts along the counit
\[\frsbit{\fosbi{\sB}}\to \sB_{\mathrm{Sh}}. \]

\begin{defn}
A formal diagram in $\sB_{\mathrm{Sh}}$ is {\bf \blacktie} if the {\punderlying} cyclic permutations of any pair of parallel maps agree.
\end{defn}
When the $\glambda_i$ are distinct, or more generally when they are aperiodic (\cref{aperiodic}), every diagram is {\blacktie}. 

\begin{thm}[Coherence for shadowed bicategories]\label{shbicat_coherence_formal}
Every {\blacktie} diagram in a shadowed bicategory commutes.
\end{thm}

Once again we proceed immediately into the proof. Fix an ordered list of composable edges $X_1,\ldots , X_n$ in a graph $\gLambda$.  Then \cref{bicat_coherence} defines a clique 
\[\bigodot_{i=1}^n X_i.\]  The image in $\frsbit{\gLambda}$ defines a clique we will denote 
\[\sh{X_1\odot \cdots \odot X_n}.\]

\begin{lem}[\cref{rmk:plan_of_proofs}\ref{rmk:plan_of_proofs_1}]\label{built_diagram_of_cliques_shadow_case}
For each $n$-tuple of composable edges $X_1,\ldots , X_n$ in a graph $\gLambda$, there is a diagram of cliques
indexed by $\sE C_n$ (\cref{present_e_cyclic_2}) where the image of $a_k$ is 
\[\sh{X_{1-k} \odot \ldots \odot X_{n-k}},\ \textup{ indices mod }n.\]
\end{lem}

\begin{proof}
For $1\leq j<n$, the rotator map $\theta$ defines a map of cliques as follows:
\[\bigsh{\left(\bigodot_{i=1}^j X_{i-k}\right)\odot \left(\bigodot_{i=j+1}^n X_{i-k}\right)} \to 
\bigsh{\left(\bigodot_{i=j+1}^n X_{i-k}\right)\odot \left(\bigodot_{i=1}^j X_{i-k}\right)} \]
(We take an admissible model for each model of the tensor products $\bigodot_{i=1}^j X_{i-k}$ and $\bigodot_{i=j+1}^n X_{i-k}$.)
Naturality of $\theta$ \eqref{theta:natural} implies this gives a well-defined map of cliques
\[\sh{X_{1-k} \odot \ldots \odot X_{n-k}}\to \sh{X_{j+1-k} \odot \ldots \odot X_{j-k}}.\]

When $j=0$, our admissible models are models where the outermost $\odot$ has only formal units $I$ on the left (and all the $X_i$ on its right), or only formal units $I$ on its right (and all $X_i$ on its left). Applying $\theta$ to each of these models defines the identity map of cliques by the shadow unit coherence \eqref{fig:shadow_unit_coherence}.
So for $j=0$ the map of cliques is the identity map.

If $k$ or $j$ is zero, the relation  $a_ka_l = a_{k+l}$ holds since $a_0$ is the identity.  If $k$ and $l$ are both nonzero, the generators $a_k$ and $a_j$ split the list in two distinct places, and rotate the resulting three segments around in different orders. Restricting to models where the last two tensor products join these segments together, we get the diagram in \ref{fig:shadow_associativity_coherence}.
\end{proof}

Let $\docf$ be the diagram of cliques in \cref{built_diagram_of_cliques_shadow_case}.  
\cref{clique_bijection} defines a functor 
\begin{equation}\label{shbicat_functor_symmetry} \int_{\sE C_n} \doc \to \comp{\frsbit{\gLambda}}{W}
\end{equation}
where $W$ is a model for $\sh{X_1 \odot \ldots \odot X_n}$. (\cref{rmk:plan_of_proofs}\ref{rmk:plan_of_proofs_2})

\begin{defn}\label{aperiodic}
A list $X_1,\ldots, X_n$ of edges of a graph $\gLambda$  is \textbf{aperiodic} if there is no nontrivial rotation of the terms that returns the same list. 
\end{defn}
Every list of distinct objects is aperiodic, but the list $X_1, X_2, X_2$ is aperiodic as well.

\begin{thm}[\cref{rmk:plan_of_proofs}\ref{rmk:plan_of_proofs_3}]\label{shadowed_bicat_coherence_aperiodic}
If the $X_i$ are aperiodic then \eqref{shbicat_functor_symmetry} is an isomorphism of categories.
\end{thm}

\begin{proof}
By construction \eqref{shbicat_functor_symmetry}  is a bijection on objects and a surjection on morphisms.
Since the source is thin, this implies it is fully faithful and therefore an isomorphism of categories.
\end{proof}

\begin{cor}[\cref{rmk:plan_of_proofs}\ref{rmk:plan_of_proofs_4}]

For each model $W$ of $\bigodot_{i=1}^n X_i$, the underlying cyclic permutation functor factors induces an equivalence of categories 
\[ \comp{\frsbit{\gLambda}}{W} \to \sB C_k, \]
where $k \mid n$ is the order of periodicity of the objects $X_i$.
\end{cor}

\begin{proof}
For aperiodic lists this is \cref{lem:consequences_of_projection_2,shadowed_bicat_coherence_aperiodic}.  
For a list with periodicity, the proof of \cref{shadowed_bicat_coherence_aperiodic} 
gives a map 
\[\int_{\sE C_n} \doc \to \comp{\frsbit{\gLambda}}{W}\]
that becomes an isomorphism once the left-hand side is quotiented out by a free action by the cyclic group $C_k$. As in \cref{cor:sym_coherence_iso}, this quotient of an abstract clique by a free $C_k$-action is equivalent to $\sB C_k$, giving the result.
\end{proof}

This finishes the proof of coherence for shadowed bicategories (\cref{shbicat_coherence_formal}).

\section{Coherence for functors}\label{sec:coherence_functor}
In this section we prove the corresponding coherence results for functors.  Since there are additional axioms and more variations in the assumptions, these proofs are elaborations of those in \cref{sec:coherence_cat}.

\subsection{Coherence for  functors of bicategories}\label{subsec:lax}

Let $\oneLax$ be the (1-)category whose objects are lax functors of bicategories $\sC \xto{F} \sD$ and whose morphisms are pairs of strict functors 
forming a strictly commuting square.
There is a forgetful functor 
\begin{equation}\label{eq:lax_forget}
\oneLax\to \Graphcat
\end{equation}
whose value on $\sC \xto{F} \sD$ is $\fobic{\sC}$  (from \cref{sec:coherence_bicat}).
The left adjoint of the functor in \eqref{eq:lax_forget} applied to a graph $\gLambda$ is a lax functor of bicategories
\[\frbic{\gLambda} \xto{\lff{\gLambda}}\lffc{ \gLambda}.\]
Here $\frbic{\gLambda}$ is the free bicategory on $\gLambda$ from \cref{sec:coherence_bicat}, and $\lffc{\gLambda}$ is the bicategory with the following presentation.

\begin{present}\label{rmk:explicit_free_lax}For a graph $\gLambda$, 
the 0-cells of the bicategory $\lffc{ \gLambda}$ are the vertices of $\gLambda$.  The 1-cells of $\lffc{\gLambda}$  are 
parenthesizations of 
\begin{equation}\label{eq:explicit_free_lax_words}W_1\odot \cdots \odot W_\ell
\end{equation}
where each $W_i$ is 
\begin{enumerate}
\item\label{item:explicit_free_lax} a 1-cell of $\frbic{\gLambda}$ written inside $\lff{\gLambda}(...)$, or
\item  a formal unit,
\end{enumerate}
and the total resulting string of edges of $G$ must be composable. A typical such word is
\begin{equation}\label{eq:typical_free_functor}
	(\lff{\gLambda}(X_1 \odot X_2) \odot \lff{\gLambda}(I \odot I)) \odot (\lff{\gLambda}(I \odot X_3) \odot I)
\end{equation}
where $X_1$, $X_2$, and $X_3$ are composable.
At this point $\lff{\gLambda}$ is notation that indicates how terms are grouped.  (Compare to \cref{ex:define_odot_clique_F}.)

The 2-cells of $\lffc{\gLambda}$ are generated by 
	\begin{enumerate}[start=1,label={\bfseries G\arabic{generators}}]
\item\stepcounter{generators} \label{rmk:explicit_free_lax_1}  the associator (\ref{rmk:explicit_free_bicat_1}) and unitor (\ref{rmk:explicit_free_bicat_2}) maps for the tensors and units  in $\frbic{\gLambda}$,
\end{enumerate}
These are applied inside the terms $\lff{\gLambda}(...)$.  There are corresponding generators for grouping these terms with each other:
	\begin{enumerate}[start=1,label={\bfseries G\arabic{generators}}]
	\item\stepcounter{generators}\label{rmk:explicit_free_lax_3}\label{rmk:explicit_free_lax_4}    the associator  (\ref{rmk:explicit_free_bicat_1})  and unitor (\ref{rmk:explicit_free_bicat_2}) maps for the tensors and units  in $\lffc{\gLambda}$.
\end{enumerate}
We add generators from a lax functor 
	\begin{enumerate}[start=1,label={\bfseries G\arabic{generators}}]
	\item\stepcounter{generators}\label{rmk:explicit_free_lax_5}  
	$i\colon I_A=I_{\lff{\gLambda}(A)} 
		\to \lff{\gLambda}(I_{A}),$ and
\item\stepcounter{generators} \label{rmk:explicit_free_lax_6} 
	$m\colon \lff{\gLambda}(W) \otimes \lff{\gLambda}(W') \to \lff{\gLambda}(W \otimes W').$
\end{enumerate}
The relations are: 
		\begin{enumerate}[start=1,label={\bfseries R\arabic{relations}}]
\item\stepcounter{relations}\label{eq:pent_tri_lax_functor_all}  the pentagon and triangle coherence  conditions (\ref{rmk:explicit_free_bicat_4}, \ref{rmk:explicit_free_bicat_4a}) for units and tensors in $\frbic{\gLambda}$ and $\lffc{\gLambda}$,
	\item\stepcounter{relations}\label{eq:unit_lax_functor_all} the coherence conditions of a lax functor relating the unit isomorphisms $i$ and $m$ 
	\[
		\xymatrix{I_A\odot \lff{\gLambda}(W)\ar[d]^\sim\ar[r]^-{i\odot \id}&\lff{\gLambda}(I_{A})\odot  \lff{\gLambda}(W)\ar[d]^{m}
			\\
			 \lff{\gLambda}(W)&\lff{\gLambda}(I_{A}\odot  W)\ar[l]_-\sim}
	\quad \text{and}
	\quad 
		\xymatrix{ \lff{\gLambda}(W)\odot I_A\ar[d]^\sim\ar[r]^-{\id\odot i}& \lff{\gLambda}(W)\odot\lff{\gLambda}(I_{A}) \ar[d]^{m}
			\\
			 \lff{\gLambda}(W)&\lff{\gLambda}(W\odot I_{A} ),\ar[l]_-\sim}
	\]
	\item\stepcounter{relations}\label{eq:hex_lax_functor} the coherence conditions of a lax functor relating the associator and $m$
		\[
			\xymatrix{(\lff{\gLambda}(W_1)\odot \lff{\gLambda}(W_2))\odot \lff{\gLambda}(W_3)\ar[r]^-{m\odot \id}\ar[d]^\sim
				&\lff{\gLambda}(W_1\odot W_2)\odot \lff{\gLambda}(W_3)\ar[r]^-m
				&\lff{\gLambda}((W_1\odot W_2)\odot W_3)\ar[d]^\sim
				\\
				\lff{\gLambda}(W_1)\odot (\lff{\gLambda}(W_2)\odot \lff{\gLambda}(W_3))\ar[r]^-{\id\odot m}
				&\lff{\gLambda}(W_1)\odot \lff{\gLambda}(W_2\odot W_3)\ar[r]^m
				&\lff{\gLambda}(W_1\odot (W_2\odot W_3)),
				}
\]
and
	\item\stepcounter{relations}\label{eq:wisk_nat_lax_functor} whiskering and naturality relations (\ref{rmk:explicit_free_bicat_5} and  \ref{rmk:explicit_free_bicat_6}).
\end{enumerate}
These relations make $\lffc{\gLambda}$  a bicategory and $\lff{\gLambda}\colon \frbic{\gLambda} \to \lffc{\gLambda}$ a lax functor. 
\end{present}

Suppose there are $n$ edges $X_1,\ldots,X_n$  of $G$  in a 1-cell  of $\lffc{\gLambda}$  and $W_{i_1}, \cdots,  W_{i_k}$
are the words of type \ref{item:explicit_free_lax} from \cref{rmk:explicit_free_lax} in the 1-cell.  
Define  a map 
\begin{equation}\label{eq:lax_functor_supporting_objects}
\alpha\colon \underline{n}\to \underline{k}
\end{equation} 
by
$\alpha(\ell)=j$ if $X_\ell\in W_{i_j}$. So for example the 1-cell depicted in \eqref{eq:typical_free_functor} is assigned to the map of totally-ordered sets
\[ \alpha(1) = 1, \quad \alpha(2) = 1, \quad \alpha(3) = 3 \]
as is any other 1-cell that matches the following picture once units and parenthesizations are ignored.
\[ \ldots \ \lff{\gLambda}( \ \ldots\  X_1 \ \ldots\  X_2 \ \ldots\  ) \ \ldots\  \lff{\gLambda}(\ \ldots\ ) \ \ldots\  \lff{\gLambda}(\ \ldots\  X_3 \ \ldots\ ) \ \ldots \]

\begin{lem}\label{eq:lax_functor_supporting_morphism}
The assignment in \eqref{eq:lax_functor_supporting_objects} extends to  a {\bf {\underlying} set} functor 
\begin{equation}\label{eq:lax_functor_underlying}
U\colon \comp{\lffc{\gLambda}(A_0,A_n)}{W}  \to (\underline{n} \downarrow \mathbf\Delta).
\end{equation}
\end{lem}

\begin{proof}
The components of $\lffc{\gLambda}$ correspond to lists of composable edges $X_1$, $\ldots$, $X_n$, and for each  component $\comp{\lffc{\gLambda}(A_0,A_n)}{W}$ we define $U$ as follows:
\begin{itemize}
	\item The images of the associator and unitor generators \ref{rmk:explicit_free_lax_1}-\ref{rmk:explicit_free_lax_4} are identity maps.
	\item  The image of the unit map generator \ref{rmk:explicit_free_lax_5} applied between groupings $i-1$ and $i$ is  the coface map (\ref{gen:coface})
\[ \xymatrix @R=1.5em{
& \underline{n} \ar[ld] \ar[rd] & \\
\underline{k} \ar[rr]^-{d^i} && \underline{k+1}.
} \]
\item The image of the  composition map generator \ref{rmk:explicit_free_lax_6} applied to groupings $i$ and $i+1$ is  the codegeneracy map (\ref{gen:codegeneracy})
\[ \xymatrix @R=1.5em{
& \underline{n} \ar[ld] \ar[rd] & \\
\underline{k} \ar[rr]^-{s^i} && \underline{k-1}.
} \]
\end{itemize}
For each of the relations \ref{eq:pent_tri_lax_functor_all}-\ref{eq:wisk_nat_lax_functor}  both branches induce the same map of sets, hence $U$ is a well-defined functor.
\end{proof}

For any lax functor $\sC \xto{F} \sD$ there exists a unique strict map of bicategories $\lffc{\fobic{\sC}}\to \sD$ so that the following square commutes.
\[\xymatrix{
	\frbic{\fobic{\sC}} \ar[d]^-{} \ar[r]^-{\lff{\fobic{\sC}}} & \lffc{\fobic{\sC}} \ar@{-->}[d]^-{\exists !} \\
	\sC \ar[r]^-F & \sD
}
\]

A {\bf formal diagram} of a lax functor $F\colon \sC\to \sD$ is any diagram in $\sD$ that lifts against the functor $\lffc{\fobic{\sC}}\to \sD$.

\begin{defn} 
A formal diagram of morphisms for a lax functor is {\bf \blacktie} if for every pair of parallel morphisms, the {\underlying} maps $U(\phi)$ for both composites agree.
\end{defn}

Note that a formal diagram for the lax functor $F$ will be {\blacktie} if every $F(\ldots)$ term contains a nontrivial object $X_i$, and not just formal units. As observed in \cite{kelly_maclane,lewis_thesis}, there is a formal diagram of the form $F(I) \rightrightarrows F(I) \odot F(I)$ that fails to commute in general.

\begin{thm}[Coherence for lax functors]\label{lax_not_coherence}
Every {\blacktie} diagram of morphisms for a lax functor commutes.
\end{thm}

\begin{rmk}\label{oplax}
	This theorem and its generalizations also hold with oplax functors instead of lax functors. The statements and constructions are the same, only the composition maps \eqref{rmk:explicit_free_lax_5} and unit maps \eqref{rmk:explicit_free_lax_6} point the other way, and the category $(\underline{n} \downarrow \mathbf\Delta)$ 
 is replaced by the opposite category $(\underline{n} \downarrow \mathbf\Delta)^\op$.
\end{rmk}

\subsubsection{Proof of  coherence for lax functors (\cref{lax_not_coherence})}\label{proof_lax_not_coherence}

\begin{lem}[\cref{rmk:plan_of_proofs}\ref{rmk:plan_of_proofs_1}]\label{built_diagram_of_cliques} For each tuple of composable edges $X_1,X_2,\ldots ,X_n$ in a graph $\gLambda$ there is a  diagram of cliques  with domain $(\underline{n} \downarrow \mathbf\Delta)$ 
and the image of a totally ordered map $\alpha\colon \underline{n} \to \underline{k}$ is the clique
\[ \bigodot_{j \in \underline{k}} \lff{\gLambda}\left( \bigodot_{i \in \alpha^{-1}(j)} X_i \right). \]
\end{lem}

This is the clique defined in 
\cref{ex:define_odot_clique_F} applied to the functor $ \lff{\gLambda}$.

\begin{proof}
\cref{built_diagram_of_cliques:lem_coherence_bicat} defines the required maps of cliques for the generators of $(\underline{n} \downarrow \mathbf\Delta)$ (see \cref{slice_presentation} and \ref{gen:coface}-\ref{gen:codegeneracy}).

	Now we check the relations \ref{faceface}-\ref{degface3}. It suffices to check each one on a single model that is admissible for all of the maps in that relation.  \ref{faceface}  follows from whiskering (\ref{rmk:explicit_free_bicat_5}) where we take a  
 model that has units $I_{A_i}$ and $I_{A_j}$ in the appropriate places and observe 
that applying \[i\colon I_{A_k} \to  \lff{\gLambda}(I_{A_k})\] to the chosen units in either order gives the same result. 
	
	\ref{degdeg} also follows from whiskering unless the codegeneracies are adjacent.  In that case we take a model with three adjacent words 
	\[(\lff{\gLambda}(W) \otimes \lff{\gLambda}(W')) \otimes \lff{\gLambda}(W'')\]
		 and apply the canonical isomorphisms and then the two codegeneracy maps.  The desired diagram becomes the hexagon from \ref{eq:hex_lax_functor}. (\cref{bicat_coherence} implies that we can take the unlabeled isomorphisms to be the associator.)
	
	\ref{degface1} to \ref{degface3} follow from  
whiskering unless the unit produced by the coface gets multiplied in by the codegeneracy.  
In this  case the admissible models are those that contain $I\odot \lff{\gLambda}(W)$ or $ \lff{\gLambda}(W)\odot I$ (with no parentheses between them). These relations then follow from \ref{eq:unit_lax_functor_all}. (\cref{bicat_coherence} implies  we can take the unlabeled isomorphisms to be the unitors.)
\end{proof}

Let $\docf$ be the diagram of cliques in \cref{built_diagram_of_cliques}.  
\cref{clique_bijection} defines a functor
\begin{equation}\label{laxfunctor_functor_symmetry}  \int_{(\underline{n} \downarrow \mathbf\Delta)} \doc \to \comp{\lffc{\gLambda}(A_0,A_n)}{W}
\end{equation}
where $W$ is any object in the component of $\bigodot_{i \in \underline{n}} \lff{\gLambda}\left(X_i \right)$. (\cref{rmk:plan_of_proofs}\ref{rmk:plan_of_proofs_2})

\begin{thm}[\cref{rmk:plan_of_proofs}\ref{rmk:plan_of_proofs_3}]\label{same_presentations}
 	The functor in \eqref{laxfunctor_functor_symmetry} is an isomorphism of categories.
\end{thm}

\begin{proof}
	By construction, \eqref{laxfunctor_functor_symmetry}  is a bijection onto the objects of $\comp{\lffc{\gLambda}(A_0,A_n)}{W}$.

\cref{groth_presentation,slice_presentation,delta_presentation} give explicit generators for  $\int_{(\underline{n} \downarrow \mathbf\Delta)} \doc$.
There is a generator for each instance of the associator isomorphisms $\alpha$ and unitor isomorphisms $l$ and $r$ applied both inside and outside $\lff{\gLambda}$ (the vertical generators), and a generator for each instance of the composition morphisms $m$ and unit morphisms $i$ (the horizontal generators). These map to all of the generators for $\comp{\lffc{\gLambda}(A_0,A_n)}{W}$ from \cref{rmk:explicit_free_lax}. Therefore \eqref{laxfunctor_functor_symmetry}  is full.
  
The composite functor
	\[ \int_{(\underline{n} \downarrow \mathbf\Delta)} \doc \xto{\eqref{laxfunctor_functor_symmetry} } \comp{\lffc{\gLambda}(A_0,A_n)}{W}\xto{\eqref{eq:lax_functor_underlying}}  (\underline{n} \downarrow \mathbf\Delta) \]
	is the projection $\pi$ to the base category from \cref{lem:consequences_of_projection_1}. Since $\pi$ is an equivalence of categories, \eqref{laxfunctor_functor_symmetry}  is faithful. Since \eqref{laxfunctor_functor_symmetry} is full, faithful, and a bijection on objects, it is an isomorphism of categories.
\end{proof}

\begin{cor}[\cref{rmk:plan_of_proofs}\ref{rmk:plan_of_proofs_4}]
The {\underlying} set functor \eqref{eq:lax_functor_underlying} is an equivalence of categories.
\end{cor}

This finishes the proof of coherence for lax functors (\cref{lax_not_coherence}).

\subsubsection{Coherence for normal lax functors}\label{subsec:normal_coherence}
The results and proofs for normal lax functors and pseudofunctors are the same as for lax functors as in \cref{proof_lax_not_coherence}, with a few small differences that we now describe.

Let  $\onenLax$ be the (1-)category whose objects are \emph{normal} lax functors of bicategories
\[\sC \xto{F} \sD\]
 and whose morphisms are pairs of strict functors forming a strictly commuting square. There is a forgetful functor 
\begin{equation}\label{eq:normal_lax_forget}
\onenLax\to \Graphcat
\end{equation}
whose value on $\sC \xto{F} \sD$ is $\fobic{\sC}$  (from \cref{sec:coherence_bicat}).
For a graph $\gLambda$, the left adjoint of the functor in \eqref{eq:normal_lax_forget} applied to $\gLambda$  is
\[\frbic{\gLambda} \xto{\nlff{\gLambda}}\nlffc{ \gLambda}\]
 where $\frbic{\gLambda}$ is as in \cref{sec:coherence_bicat}. The presentation of $\nlffc{ \gLambda}$ is the same as in \cref{rmk:explicit_free_lax}, except the unitor maps \ref{rmk:explicit_free_lax_5} are now invertible generators.

\begin{lem}\label{eq:normal_functor_supporting_morphism}
The assignment in \eqref{eq:lax_functor_supporting_objects} extends to a {\bf {\underlying} set} functor 
\begin{equation}\label{eq:normal_lax_functor_underlying}
U\colon \comp{\nlffc{\gLambda}(A_0,A_n)}{W}  \to (\underline{n} \downarrow \mathbf\Delta)[\sI^{-1}].
\end{equation}
\end{lem}

\begin{proof}
The construction is as in \cref{eq:lax_functor_supporting_morphism}, except that the slice category has been localized by inverting the injective maps $\sI$, so that each of the unitor maps \ref{rmk:explicit_free_lax_5} can be sent to an isomorphism.
\end{proof}

For any normal lax functor $\sC \xto{F} \sD$ there exists a unique strict map of bicategories $\nlffc{\fobic{\sC}}\to \sD$ so that the following square commutes.
\[\xymatrix{
	\frbic{\fobic{\sC}} \ar[d]^-{} \ar[r]^-{\nlff{\fobic{\sC}}} & \nlffc{\fobic{\sC}} \ar@{-->}[d]^-{\exists !} \\
	\sC \ar[r]^-F & \sD
} \]

A {\bf formal diagram} of a lax functor $F\colon \sC\to \sD$ is any diagram in $\sD$ that lifts against the functor $\nlffc{\fobic{\sC}}\to \sD$. In contrast to the case of lax functors, we do not need to impose additional conditions on formal diagrams.

\begin{thm}[Coherence for normal (op)lax functors]\label{strong_coherence_normal_lax}
All formal diagrams for a normal (op)lax functor  $F$ commute.
\end{thm}

We can now start reusing ideas from  \cref{proof_lax_not_coherence}.

\begin{lem}[\cref{rmk:plan_of_proofs}\ref{rmk:plan_of_proofs_1}, compare to \cref{built_diagram_of_cliques}]\label{built_diagram_of_cliques_normal} For each tuple of composable edges $X_1,X_2,\ldots ,X_n$ in a graph $\gLambda$ there is a  diagram of cliques  with domain $(\underline{n} \downarrow \mathbf\Delta)[\sI^{-1}]$ 
and the image of a totally ordered map $\alpha\colon \underline{n} \to \underline{k}$ is the clique
\[ \bigodot_{j \in \underline{k}} \nlff{\gLambda}\left( \bigodot_{i \in \alpha^{-1}(j)} X_i \right). \]
\end{lem}

\begin{proof}
The proof is the same as \cref{built_diagram_of_cliques} with two exceptions. Since the slice category $(\underline{n} \downarrow \mathbf\Delta)$ is localized at the injective maps $\sI$, its presentation is changed -- the coface maps are now invertible generators, while the codegeneracy maps are still ordinary generators. But the coface maps are sent to the unitor maps \ref{rmk:explicit_free_lax_5}, which are isomorphisms because we are now dealing with normal lax functors. The rest of the verification proceeds as in \cref{built_diagram_of_cliques}.
\end{proof}

Let $\docf$ be the diagram of cliques in \cref{built_diagram_of_cliques_normal}.  
\cref{clique_bijection} defines a functor (\cref{rmk:plan_of_proofs}\ref{rmk:plan_of_proofs_2})
\begin{equation}\label{laxfunctor_functor_symmetry_normal}  \int_{(\underline{n} \downarrow \mathbf\Delta)[\sI^{-1}]} \doc \to \comp{\nlffc{\gLambda}(A_0,A_n)}{W}
\end{equation}
where $W$ is a model for $\bigodot_{i=1}^n \lff{\gLambda}\left( X_i \right)$.

\begin{thm}[\cref{rmk:plan_of_proofs}\ref{rmk:plan_of_proofs_3}]\label{same_presentations_normal}
 	The functor in \eqref{laxfunctor_functor_symmetry_normal} is an isomorphism of categories.
\end{thm}
The proof of \cref{same_presentations_normal} is the same as the proof of \cref{same_presentations}.

\begin{cor}[\cref{rmk:plan_of_proofs}\ref{rmk:plan_of_proofs_4}]\label{lax_coherence_normal}
The {\underlying} set functor \eqref{eq:normal_lax_functor_underlying} is an equivalence of categories.
 \end{cor}

This shows that formal diagrams are equivalent to the localized comma category $(\underline{n} \downarrow \mathbf\Delta)[\sI^{-1}]$. Our goal is to prove that all formal diagrams commute, so it remains to show:

\begin{lem}\label{localized_comma_cat_is_thin}
	The localization $(\underline{n} \downarrow \mathbf\Delta)[\sI^{-1}]$ is a thin category.
\end{lem}

\begin{proof}
	Each object $f\colon \underline{n} \to \underline{k}$ in the localization is isomorphic to one in which $f$ is surjective. Along this isomorphism, each zig-zag of morphisms in the comma category (with backwards morphisms injective) becomes a zig-zag between objects with $f$ surjective. Between two such objects, the injective maps are bijective, so each zig-zag simplifies to a single morphism. But between objects with $f$ surjective, any ordered pair of such objects admits at most one morphism between them, so the category is thin.
\end{proof}

This finishes the proof of coherence for normal lax functors (\cref{strong_coherence_normal_lax}).

\subsubsection{Coherence for pseudofunctors}
The similarities for coherence for pseudofunctors and normal lax functors is even stronger than those between coherence for normal lax functors and lax functors.  In \cref{subsec:normal_coherence} replace $\onenLax$ by the category $\onepLax$ whose objects are pseudofunctors and whose morphisms are commuting squares of strict functors.   Let 
\[\frbic{\gLambda} \xto{\pff{\gLambda}}\pffc{ \gLambda}\]
be the value of the free pseudofunctor on a graph $\gLambda$.

\begin{thm}[Coherence for pseudofunctors]\label{strong_coherence}
All formal diagrams for a pseudofunctor $F$ commute.
\end{thm}

In \cref{built_diagram_of_cliques_normal} both the associator and unitor maps are isomorphisms and so the indexing category is the localization of the slice category $(\underline{n} \downarrow \mathbf\Delta)$ by all morphisms. This is a clique because it is a groupoid with an initial object. With this modification, the argument in \cref{subsec:normal_coherence} implies 

\begin{cor}[\cref{rmk:plan_of_proofs}\ref{rmk:plan_of_proofs_4}]\label{lax_coherence_pseudo}
For each component of $\pffc{\gLambda}(A_0,A_n)$, the supporting set functor
\[ \comp{\pffc{\gLambda}(A_0,A_n)}{W} \to (\underline{n} \downarrow \mathbf\Delta)[(\underline{n} \downarrow \mathbf\Delta)^{-1}] \]
is an equivalence of categories. Therefore $\pffc{\gLambda}(A_0,A_n)$ is thin.
 \end{cor}

This is enough to prove coherence for pseudofunctors.

\subsubsection{A more general theorem for pseudofunctors}
The coherence theorem for pseudofunctors (\cref{strong_coherence}) also has a more general statement involving the free pseudofunctor on a \emph{map} of graphs, rather than a single graph.

The category $\onepLax$ of pseudofunctors and strict maps between them admits a forgetful functor to the arrow category of graphs. Its left adjoint sends each map of graphs $H\colon \gLambda \to M$ to a pseudofunctor
\[ \frbic{\gLambda} \xto{\fmfrbic{H}} \mfrbic{H} \]
with the following property: for any pseudofunctor $\sC \xto{F} \sD$ and strictly commuting square of graphs in \cref{fig:sgFtoD_1}
there exist unique strict maps of bicategories (the dashed maps) is \cref{fig:sgFtoD_2} making the square in \cref{fig:sgFtoD_2} and all regions in \cref{fig:sgFtoD_3} commute strictly.
\begin{figure}[h]
	\hspace{1cm}
     \centering
     \begin{subfigure}[b]{0.18\textwidth} 
		\centering 
\xymatrix{
	\gLambda \ar[d] \ar[r]^-H &M \ar[d] \\
	\fobic{\sC} \ar[r]^-{\fobic{F}} & \fobic{\sD}
} 
	\caption{Graph maps}\label{fig:sgFtoD_1}
     \end{subfigure}
     \hfill
     \begin{subfigure}[b]{0.2\textwidth}
         \centering
\xymatrix{
	\frbic{\gLambda} \ar@{-->}[d]^-{\exists !} \ar[r]^-{\fmfrbic{H}} & \mfrbic{H} \ar@{-->}[d]^-{\exists !} \\
	\sC \ar[r]^-F & \sD
} 
	\caption{Pseudofunctors}\label{fig:sgFtoD_2}
     \end{subfigure}
     \hfill
     \begin{subfigure}[b]{0.25\textwidth}
         \centering
\xymatrix{
	\gLambda \ar@/_3em/[dd] \ar[d] \ar[r]^-H & M \ar[d] \ar@/^3em/[dd] \\
	\fobic{\frbic{\gLambda}} \ar@{-->}[d] \ar[r]^-{\fobic{\fmfrbic{H}}} & \fobic{\mfrbic{ H}} \ar@{-->}[d] \\
	\fobic{\sC} \ar[r]^-{\fobic{F}} & \fobic{\sD}
} 
	\caption{Compatibility}\label{fig:sgFtoD_3}
     \end{subfigure}
	\hspace{1cm}
        \caption{Definition of the functor $\mfrbic{H} \to \sD$}\label{fig:sgFtoD}
\end{figure}

In particular, there is a functor $\mfrbic{H} \to \sD$ for each pseudofunctor $F\colon \sC\to \sD$. An {\bf extended formal diagram} for $F$ is any diagram in $\sD$ that lifts against the functor $\mfrbic{H} \to \sD$.

\begin{thm}[Coherence for pseudofunctors, relative version]\label{strong_coherence_extended}
All extended formal diagrams for a pseudofunctor $F$ commute.
\end{thm}

\begin{present}\label{rmk:present_prelim_free_lax_map}
Before the presentation for $\mfrbic{H}$, we first describe a closely related bicategory $\prelimmfrbic{H}$ that we need for the proof of \cref{strong_coherence}.

The 0-cells of $\prelimmfrbic{H}$ are vertices of $M$.  The 1-cells of $\prelimmfrbic{H}$  are 
parenthesizations of 
\[W_1\odot \cdots \odot W_\ell\]
as in \cref{rmk:explicit_free_lax}, but where each $W_i$ is 
\begin{enumerate}
\item\label{item:explicit_free_lax_stronger} a 1-cell of $\frbic{\gLambda}$ written inside $\fmfrbic{H}(...)$,
\item\label{item:explicit_free_lax_stronger_3}  a formal unit, \emph{or} 
\item\label{item:explicit_free_lax_stronger_2}   an edge of $M$, 

\end{enumerate}
and the total resulting string of edges must be composable, in the sense that adjacent edges in $\gLambda$ are composable, and when $H$ is applied every edge in $\gLambda$, the resulting edges in $M$ are all composable.

The generators are \ref{rmk:explicit_free_lax_1} to \ref{rmk:explicit_free_lax_6} with the additional assumption that $m$ and $i$ are invertible generators (\ref{rmk:explicit_free_lax_5} and \ref{rmk:explicit_free_lax_6}).  The relations are \ref{eq:pent_tri_lax_functor_all} to \ref{eq:wisk_nat_lax_functor}. 
\end{present}

Each component of $\prelimmfrbic{H}$ is associated to a tuple
\begin{equation}\label{eq:H_0_indexing}
((Y_1,\ldots ,Y_n),S,\{X_i\}_{i\in S})
\end{equation}
satisfying the following conditions.
\begin{enumerate}
\item  $Y_1,\ldots ,Y_n$ are a string of composable edges in $M$, 
\item $S$ is a possibly empty subset of $\{1,\ldots n\}$, and 
\item $X_i$ is a choice of preimage of $Y_i$ for each $i\in S$ such that   adjacent preimages $X_i$, $X_{i+1}$ are composable in $G$.
\end{enumerate}
(There is not necessarily a bijection between the edges $Y_i$ and words $W_i$ as in \cref{rmk:present_prelim_free_lax_map}!)

Fix one such component. Partition $S$ into its maximal consecutive subsets $\underline{n_j} \subset S$, $j \in J$. As these are subsets of $\underline{n}$, each one is a totally ordered set and so the comma category $(\underline{n_j} \downarrow \mathbf\Delta)$ can be defined.

For each collection of totally ordered maps $\{\alpha_j\colon \underline{n_j} \to \underline{k_j}\}_{j \in J}$, we define $\underline{k}$ to be the union of all of the $\underline{k_j}$ and the set $\underline{n} \setminus S$. The set $\underline{k}$ has an total order induced by the orders on $\underline{n}$, $J$, and the $\underline{k_j}$.
Define $\alpha\colon \underline{n} \to \underline{k}$ to be the maps $\alpha_j$ on each subset $\underline{n_j}$, and the identity of $\underline{n} \setminus S$ otherwise.

Fixing one set of maps $\{\alpha_j\}_{j \in J}$, we take all 1-cells in $\prelimmfrbic{H}$ obtained by
\begin{itemize}
	\item taking the edges $Y_i$ with $i \in \underline{n} \setminus S$ (terms of type \ref{item:explicit_free_lax_stronger_2} from \cref{rmk:present_prelim_free_lax_map}),
	\item for each $j \in J$ and $\ell \in \underline{k_j}$, taking a 1-cell of $\frbic{\gLambda}$ on the edges $\{X_i\}_{i \in \alpha_j^{-1}(\ell)}$ and applying $\fmfrbic{H}(...)$ (terms of type \ref{item:explicit_free_lax_stronger} from \cref{rmk:present_prelim_free_lax_map}), and
	\item any finite number of formal units (terms of type \ref{item:explicit_free_lax_stronger_3} from \cref{rmk:present_prelim_free_lax_map}).
\end{itemize}
We arrange the terms to respect the ordering in $
	\underline{n}$ and we include  all parenthesization.
	These form a clique by taking the associators and unitors from $\frbic{\gLambda}$ inside the groupings $\fmfrbic{H}(...)$ and on the total word. In other words, the clique is a product
	\begin{equation}\label{eq:H_0_define_clique}
	\ob\left(\bigodot_{\ell \in \underline{k}}Z_\ell\right)\times \prod_{j\in J}\prod_{\ell\in\underline{k_j}}\ob \left( \bigodot_{i \in \alpha^{-1}(\ell)} X_i \right)
	\end{equation}
	where the $Z_\ell$ are placeholders as in \cref{ex:define_odot_clique_F}.
	
	\begin{example}
Consider a 1-cell of the form
\[\fmfrbic{H}(X_1\odot X_2\odot I\odot X_3)\odot I\odot \fmfrbic{H}( X_4)\odot Y_5\odot \fmfrbic{H}(X_6\odot I\odot X_7)\]
with additional parentheses not drawn. The $X_i$ are edges in $\gLambda$ and $Y_5$ is an edge in $M$. This 1-cell arises from the clique in which
$\underline{n} = \{1,2,3,4,5,6,7\}$,
$S = \{1,2,3,4,6,7\}$,
$J=\{1,2\}$,
$\underline{p_1}=\{1,2,3,4\}$, $\underline{p_2}=\{6,7\}$,
$\underline{k} = \{1,2,3,4\}$,
$\underline{k_1}=\{1,2\}$, $\underline{k_2}=\{4\}$.  The function 
\[ \alpha_1\colon \underline{p_1}\to \underline{k_1}\] is defined by $\alpha_1(1)=\alpha_1(2)=\alpha_1(3)=1$ and $\alpha_1(4)=2$.  The  function 
\[ \alpha_2\colon \underline{p_2}\to \underline{k_2}\] is defined by $\alpha_2(6)=\alpha_2(7)=4$. 
\end{example}

We get such a clique for each object in the product category
\[ \prod_j (\underline{n_j} \downarrow \Delta)[\Delta^{-1}]. \]
We define maps between the cliques in \eqref{eq:H_0_define_clique} by sending the coface and codegeneracy maps in each of the categories $(\underline{n_j} \downarrow \Delta)$ to instances of $i$ and $m$, respectively.

\begin{lem}[\cref{rmk:plan_of_proofs}\ref{rmk:plan_of_proofs_1}]\label{built_diagram_of_cliques_maps_prelim} 
The cliques in \eqref{eq:H_0_define_clique} and the maps induced by  $m$ and $i$ define a diagram of cliques.
\end{lem}

The proof is the same as in \cref{built_diagram_of_cliques}.  Let $\docf$ be the diagram of cliques in \cref{built_diagram_of_cliques_maps_prelim}.  
\cref{clique_bijection} defines a functor
\begin{equation}\label{laxfunctor_functor_symmetry_maps_prelim}  \int_{\prod_j(\underline{n_j} \downarrow \mathbf\Delta)[\mathbf\Delta^{-1}]} \doc \to \comp{\prelimmfrbic{H}(A_0,A_n)}{W}.
\end{equation}

\begin{lem}[\cref{rmk:plan_of_proofs}\ref{rmk:plan_of_proofs_3}]\label{same_presentations_maps_prelim}
 	The functor in \eqref{laxfunctor_functor_symmetry_maps_prelim} is an isomorphism of categories and therefore the components of $\prelimmfrbic{H}(A_0,A_n)$ are cliques.
\end{lem}

\begin{proof}
The map \eqref{laxfunctor_functor_symmetry_maps_prelim} is bijective on objects and surjective on morphisms, and the source is an abstract clique, so it is an isomorphism of categories.
\end{proof}

\begin{present}\label{rmk:present_free_lax_map}
The presentation for $\mfrbic{H}$ is the same as \cref{rmk:present_prelim_free_lax_map} with the additional \emph{identification} that 
for any edge $X_i$ in $\gLambda$, any term of $\fmfrbic{H}(X_i)$ in a 1-cell can be substituted for $Y_i = H(X_i)$.  (This is \emph{not} an additional isomorphism.) In addition, each generator that does not combine this term with another using $m$, or insert a new unit inside this particular term using $\ell$ or $r$, is also identified to the corresponding generator after the substitution $\fmfrbic{H}(X_i) = Y_i$ is made.
\end{present}

There is a functor
\begin{equation}\label{glue_along_identifications}
	Q\colon \prelimmfrbic{H}\to \mfrbic{H}
\end{equation}
that is the identity on 0-cells and identifies 1- and 2-cells with their images under the additional identification in \cref{rmk:present_free_lax_map}. Fix a string of composable edges $Y_1$, $\ldots$, $Y_n$ in $M$. The component of $Y_1\odot Y_2\odot \cdots \odot Y_n$ in $\mfrbic{H}(A_0,A_n)$ has as its preimage the components corresponding to all choices of $S \subseteq \underline{n}$ and choices of composable $X_i$ for $i \in S$, as in \eqref{eq:H_0_indexing}.

\begin{lem}\label{lem:partial_order}
There is a partial order on the preimage of $Y_1\odot Y_2\odot \cdots \odot Y_n$ 
where 
\[(T,\{W_i\}_{i\in T})< (S,\{X_i\}_{i\in S})\]
if $T\subset S$ and $W_i=X_i$ for all $i\in T$.
\end{lem}

Our goal is to show that the component $\comp{\mfrbic{H}(A_0,A_n)}{Y_1 \odot \cdots \odot Y_n}$ is a clique. Without loss of generality, $H$ consists of edges $Y_1, \cdots, Y_n$ and $G$ has finitely many edges in total. Therefore the partial ordering in \cref{lem:partial_order} is on a finite set. Extend it to a total ordering and let $\bC_{> (S,\{X_i\}_{i\in S})}$ be the category constructed by identifying the presentations of those components of $\prelimmfrbic{H}$ corresponding to tuples higher in the ordering than $(S,\{X_i\}_{i\in S})$. 

Inductively, assume we have shown that $\bC_{> (S,\{X_i\}_{i\in S})}$ is a clique and we wish to show that $\bC_{\geq (S,\{X_i\}_{i\in S})}$ is a clique.

We now formalize the process of adding a component of $\prelimmfrbic{H}$ to $\bC_{> (S,\{X_i\}_{i\in S})}$ giving us $\bC_{\geq (S,\{X_i\}_{i\in S})}$. Let $\bC$ be a category with a presentation, $O_{\bC}$ be a nonempty subset of the objects of $\bC$ and $G_{\bC}$ be a subset of the generators of $\bC$.  Let $(O_{\bC},G_{\bC})$ be subcategory of $\bC$ on the objects of $O_{\bC}$ generated by $G_{\bC}$.   Let $\bD$ be another such category with nonempty subset of objects $O_{\bD}$ and generators $G_{\bD}$.  
Given compatible bijections $\alpha\colon O_{\bC}\to O_{\bD}$ and $\beta\colon G_{\bC}\to G_{\bD}$, define 
a category $G(\alpha,\beta)$ with 
\begin{itemize}
\item objects the pushout of the objects of $\bC$ and those of $\bD$ along the bijection $O_{\bC}\to O_{\bD}$
\item generators the pushout of the morphisms  of $\bC$ and those of $\bD$ along the bijection $G_{\bC}\to G_{\bD}$, and
\item all relations in $\bC$ and $\bD$.
\end{itemize}

\begin{lem}\label{lem:induction_subcat}
For categories with presentations and bijections as above, if 
\begin{enumerate}
\item $\bC$ and $\bD$ are cliques, and
\item $(O_{\bD},G_{\bD})$ generates a full subcategory of $\bD$, 
\end{enumerate}
then $G(\alpha,\beta)$ is a clique.
\end{lem}

\begin{proof}
	The category $G(\alpha,\beta)$ is a groupoid since every generator has an inverse. Let $X$ be any object in $G(\alpha,\beta)$ lying in the identified sets of objects $O_\bC \cong O_\bD$. It suffices to show that for any other object $Y$ in $G(\alpha,\beta)$ there is a unique morphism from $X$ to $Y$.
	
	Without loss of generality $Y$ is in the object set of $\bC$. Any morphism from $X$ to $Y$ can be written as a product of generators from $\bC$ and $\bD$. Each string of consecutive generators in $\bD$ begins and ends in the identified objects $O_\bC \cong O_\bD$. Therefore it can be written in terms of the generators $G_\bD$ (this is where we use the full subcategory assumption). Replacing those generators by the corresponding ones in $G_{\bC}$, the morphism from $X$ to $Y$ agrees with the unique such morphism in the category $\bC$.
\end{proof}

\begin{lem}
They hypotheses of \cref{lem:induction_subcat} are satisfied in the identification of $\bC_{> (S,\{X_i\}_{i\in S})}$ with $\comp{\mfrbic{H}(A_0,A_n)}{(S,\{X_i\}_{i\in S})}$ to form $\bC_{\geq (S,\{X_i\}_{i\in S})}$.
\end{lem}

\begin{proof}
	The common subset of objects consists of those objects in $\comp{\prelimmfrbic{H}(A_0,A_n)}{(S,\{X_j\}_{j \in S})}$ in which there exists a term of the form $\fmfrbic{H}(X_i)$ for some $i$. This is identified to the corresponding object in $\bC_{> (S,\{X_j\}_{j\in S})}$ in which $i$ is removed from the set $S$ and $\fmfrbic{H}(X_i)$ is replaced by $Y_i$. The common generators are those that make sense if $\fmfrbic{H}(X_i)$ is replaced by $Y_i$. 
	
	It suffices to show that two such objects (with possibly different values of $i$) can be connected by a map inside this subcategory. We work inside $\bC_{> (S,\{X_i\}_{i\in S})}$. Suppose $V$ contains $Y_i$ and $W$ contains $Y_j$. If $V$ does not contain $Y_j$, we apply unitor and $m^{-1}$ maps to isolate $X_j$ in a term by itself $\fmfrbic{H}(X_j)$.  Then we replace $\fmfrbic{H}(X_j)$ by $Y_j$.  This composite of maps is in the desired subcategory because we do not change $Y_i$ to do this. After similarly changing $W$ to contain $Y_i$, the two can then be connected by an isomorphism inside the clique $\comp{\prelimmfrbic{H}(A_0,A_n)}{(S \setminus \{i,j\},\{X_k\})}$.
\end{proof}

After finitely many steps of the induction, we conclude that

\begin{cor}
 $\comp{\mfrbic{H}(A_0,A_n)}{Y_1 \odot \cdots \odot Y_n}$ is a clique.
\end{cor}

This finishes the proof of \cref{strong_coherence_extended}.

\subsection{Symmetric monoidal functors}
We follow \cref{subsec:lax} and first consider lax symmetric monoidal functors. Then we describe the modifications to the proofs to apply them to normal lax functors and strong symmetric monoidal functors.

Let $\oneLMon$ be the (1-)category whose objects are lax symmetric monoidal functors and 
whose morphisms are pairs of strict functors forming a strictly commuting square. There is a forgetful functor 
\begin{equation}\label{eq:lax_symmetric_forget} 
\oneLMon\to \Set
\end{equation}
whose value on $\bC \xto{F} \bD$ is $\fosmc{\bC}$.
This forgetful functor has a left adjoint and the image of a set $\sLambda$ under this free functor will be denoted $\frsmc{\sLambda} \xto{\lsff{\sLambda}}{\lsffc{\sLambda}}  $.

\begin{present}\label{present:lax_symmetric}The objects of  $ \lsffc{\sLambda} $ are the same as for lax functors (\cref{rmk:explicit_free_lax}).  The generators are \ref{rmk:explicit_free_lax_1}  
to \ref{rmk:explicit_free_lax_5}   and 
	\begin{enumerate}[start=1,label={\bfseries G\arabic{generators}}]
\item\stepcounter{generators}\label{sym_functor_gen_1}
	$\gamma_o\colon \lsff{\sLambda} (W) \otimes \lsff{\sLambda}(W') \to  \lsff{\sLambda}(W') \otimes \lsff{\sLambda}(W) $
	\item\stepcounter{generators}\label{sym_functor_gen_2}
	$\gamma_i\colon  \lsff{\sLambda}(W \otimes W')\to \lsff{\sLambda}(W' \otimes W).$
\end{enumerate}
The relations are
\ref{eq:pent_tri_lax_functor_all} to \ref{eq:wisk_nat_lax_functor}  and 
\begin{enumerate}[start=1,label={\bfseries R\arabic{relations}}]
\item\stepcounter{relations} \label{lax:monoidal_relation}
	\[ \xymatrix{
		\lsff{\sLambda} (W) \otimes \lsff{\sLambda} (W') \ar@{<->}[d]^-{\gamma_o} \ar[r]^-m & \lsff{\sLambda} (W \otimes W') \ar@{<->}[d]^-{\gamma_i} \\
		\lsff{\sLambda} (W') \otimes \lsff{\sLambda} (W) \ar[r]^-m & \lsff{\sLambda} (W' \otimes W).
	} \]
\end{enumerate}

\end{present}

\begin{lem}
The construction in \cref{eq:lax_functor_supporting_morphism} extends to define a  {\bf {\underlying} set} functor 
\begin{equation}\label{eq:symmetric_underlying_functor}
 U\colon \comp{\lsffc{\sLambda}}{W} \to 
(\underline{n} \downarrow \mathbf{Fin}).
\end{equation}
\end{lem}

\begin{proof}

The components of $\lsffc{\sLambda}$ correspond to lists of elements  $\slambda_1$, $\ldots$, $\slambda_n$.  Then $U$ is defined as follows.
\begin{itemize}
\item The image of a word is the map of finite sets represented by the grouping of the $\slambda_i$ terms inside the $\lsff{\sLambda}(...)$.
\item The  images of  $m$ and $i$  are codegeneracy and coface maps.
\item The image of  $\gamma_o$ from \ref{sym_functor_gen_1} is the corresponding  transposition map $\underline{k} \to \underline{k}$.
\item The image of  $\gamma_i$ in \ref{sym_functor_gen_2} is an identity map. 
\end{itemize}
The only additional relation to check is \ref{lax:monoidal_relation}, which goes to a commuting map of sets.
\end{proof}

\begin{lem}
The {\punderlying} permutation functor from \cref{lem:underlying_functor_symmetric} extends to an {\bf {\punderlying} permutation} functor
\begin{equation}\label{eq:symmetric_permutation_functor}
 P\colon \comp{\lsffc{\sLambda}}{W} \to 
\sB\Sigma_n.
\end{equation}
\end{lem}
\begin{proof}
The images of $\gamma_o$ in \ref{sym_functor_gen_1} and $\gamma_i$ in  \ref{sym_functor_gen_2} are the corresponding permutation of the terms $\slambda_1$, $\ldots$, $\slambda_n$, and the images of  all other generators are the identity.
\end{proof}

A {\bf formal diagram} of a lax functor $F\colon \bC\to \bD$ is any diagram in $\bD$ that lifts against the functor $\lsffc{\fols{\bC}}\to \bD$.

\begin{defn} 
A formal diagram of morphisms for a lax functor is {\bf \blacktie} if the {\underlying} maps and {\punderlying} permutations for both composites are the same.
\end{defn}

Note that if the terms $\slambda_1$, $\ldots$, $\slambda_n$ are distinct then any two formal maps with the same source and target must give the same permutation, so $P$ can be safely ignored.

\begin{thm}[Coherence for (op)lax symmetric monoidal functors]\label{symmetric_not_coherence}
Every {\blacktie} diagram of morphisms for a lax symmetric monoidal functor commutes.
\end{thm}

As in \cref{oplax}, the same is true for oplax symmetric monoidal functors, with essentially the same proof.

\subsubsection{Proof of coherence for lax symmetric monoidal functors}\label{sec:proof_coherence_lax_smc_functors}

If the {\punderlying} permutations of two formal composites are the same, then those composites can be interpreted as acting on a list of terms $\slambda_1$, $\ldots$, $\slambda_n$ that are distinct. Hence, without loss of generality, we can ignore \eqref{eq:symmetric_permutation_functor} and focus on the case where the elements $\slambda_i$ are distinct.

\begin{lem}[\cref{rmk:plan_of_proofs}\ref{rmk:plan_of_proofs_1}]\label{built_diagram_of_cliques_symmetric_case}
	There is a  diagram of cliques from $(\underline{n} \downarrow \mathbf{Fin})$ to $\lsffc{\sLambda}$ where the image of $\alpha$ is the clique in \cref{building_block_clique_symmetric_monoidal_functor}.
\end{lem}

\begin{proof}

Recall from \cref{fin_presentation} that $\mathbf{Fin}$ is generated by permutations, and coface and codegeneracy maps. By \cref{cor:sym_coherence_iso}, each permutation $\sigma \in \Sigma_k$ gives a map from the clique for $\alpha$ to the clique for $\sigma\circ\alpha$. The coface and codegeneracy maps from $\mathbf\Delta$ also give clique maps by the argument in the proof of \cref{built_diagram_of_cliques}  but now using \cref{cor:sym_coherence_iso}.

	For the relations from \cref{fin_presentation} we check that each word in the relation gives the same map of cliques. For \ref{faceface} to \ref{degface3}, this is by the proof of \cref{built_diagram_of_cliques}. For \ref{squaretozero}  to \ref{threecycle}, this follows from \cref{cor:sym_coherence_iso}.
	
	For  \ref{relation:fin_coequalizer}, it is enough to consider the case where $\sigma$ is a transposition.  We can further reduce to the case where $\alpha$ is a codegeneracy map folding the two transposed points into one. We take a model with two adjacent words $\lsff{\sLambda}(W) \otimes \lsff{\sLambda}(W')$. The relation becomes the square in \ref{lax:monoidal_relation}.
The vertical map on the right is by definition any canonical isomorphism, but we can take it to be $\gamma$.
	
	For \ref{relation:fin_swap}, it is enough to consider the case where $\alpha$ is a coface or codegeneracy map. If $\alpha$ is a coface then it follows by naturality of $\gamma$ (\ref{expand_rep_symmetric_relation_4}). If $\alpha$ is a codegeneracy then it follows by the diagram
	\[\xymatrix @R=1.8em{\lsff{\sLambda}(W) \otimes (\lsff{\sLambda}(W'') \otimes \lsff{\sLambda}(W') )\ar[r]^-\alpha\ar[d]^{1\otimes \gamma}
&(\lsff{\sLambda}(W) \otimes \lsff{\sLambda}(W'')) \otimes \lsff{\sLambda}(W') \ar[d]^{\gamma\otimes 1}
\\
\lsff{\sLambda}(W) \otimes (\lsff{\sLambda}(W') \otimes \lsff{\sLambda}(W'') )\ar[d]^-\alpha
&(\lsff{\sLambda}(W'') \otimes \lsff{\sLambda}(W')) \otimes \lsff{\sLambda}(W') \ar[d]^-\alpha
\\
(\lsff{\sLambda}(W) \otimes \lsff{\sLambda}(W')) \otimes \lsff{\sLambda}(W'') \ar[r]^-\gamma\ar[d]^-{m\otimes 1}
&\lsff{\sLambda}(W'') \otimes (\lsff{\sLambda}(W') \otimes \lsff{\sLambda}(W'))  \ar[d]^{1\otimes m}
\\
\lsff{\sLambda}(W \otimes W') \otimes \lsff{\sLambda}(W'') \ar[r]^-\gamma
&\lsff{\sLambda}(W'') \otimes \lsff{\sLambda}(W' \otimes W') 
}\]
that commutes by \cref{cor:sym_coherence_iso} and  \ref{expand_rep_symmetric_relation_4}.  
\end{proof}

Let $\doc$ be the diagram of cliques in \cref{built_diagram_of_cliques_symmetric_case}.  
\cref{clique_bijection} defines a functor (\cref{rmk:plan_of_proofs}\ref{rmk:plan_of_proofs_2})
\begin{equation}\label{eq:sym_mon_cat_iso}  \int_{(\underline{n} \downarrow \mathbf{Fin})} \doc \to \comp{\lsffc{\sLambda}}{W}
\end{equation}
where $W$ is a model for $\bigotimes_{i=1}^n \lsff{\sLambda}\left( X_i \right)$.

\begin{thm}[\cref{rmk:plan_of_proofs}\ref{rmk:plan_of_proofs_3}]\label{same_presentations_symmetric}
 	The functor in \eqref{eq:sym_mon_cat_iso} is an isomorphism of categories.
\end{thm}

\begin{proof}
	By construction, \eqref{eq:sym_mon_cat_iso}  is a bijection onto the objects of $\comp{\lsffc{\sLambda}}{W}$. (If the elements $\slambda_i$ were not distinct then this claim would fail.)

The generators of $\int_{(\underline{n} \downarrow \mathbf{Fin})} \doc$ given by \cref{groth_presentation,fin_presentation} are the generators  in each clique, together with the cofaces, codegeneracies, and transpositions. These correspond to the generators of $\comp{\lsffc{\sLambda}}{W}$ (the cliques giving all expanded instances of $\alpha$, $l$, $r$, and $\gamma$, save for $\gamma$ on the outside, and the horizontal generators giving $i$, $m$, and the instances of $\gamma$ on the outside).  	Therefore this functor is  surjective on morphisms.
	
The composite functor
\[ \int_{(\underline{n} \downarrow \mathbf{Fin})} \doc \xto{\eqref{eq:sym_mon_cat_iso}}  \comp{\lsffc{\sLambda}}{W}
\xto{\eqref{eq:symmetric_underlying_functor}} (\underline{n} \downarrow \mathbf{Fin})\]
is the projection $\pi$ to the base category from \cref{lem:consequences_of_projection_1}.  By that result, $\pi$ is an equivalence of categories and so \eqref{eq:symmetric_underlying_functor} is faithful.

Since \eqref{eq:symmetric_underlying_functor} is an isomorphism on objects and full and faithful it is an isomorphism.
\end{proof}

\begin{cor}[\cref{rmk:plan_of_proofs}\ref{rmk:plan_of_proofs_4}]\label{symmetric_coherence}
When the elements $\slambda_i$ are distinct, the {\underlying} set functor \eqref{eq:symmetric_underlying_functor} is an equivalence of categories.
 \end{cor}

This finishes the proof of coherence for lax symmetric monoidal functors (\cref{symmetric_not_coherence}).

\subsubsection{Normal and strong symmetric monoidal functors}
As with functors of bicategories, the cases of normal and strong monoidal functors follow in almost exactly the same way as the case for lax monoidal functors.  

For normal functors, replace the category $\oneLMon$ with the corresponding category for lax normal functors with a forgetful functor 
\[
\oneNLMon\to \Set\]
and let \[\frsmc{\sLambda} \xto{\nlsff{\sLambda}}{\nlsffc{\sLambda}} \] be the result of applying the free functor to a set $\sLambda$.   Then the presentation for $\nlsffc{\sLambda}$ is as in 
\cref{present:lax_symmetric} except that the unit maps $i$ are invertible generators. The {\underlying} set functor  goes from a component of $\nlsffc{\sLambda}$ to $(\underline{n} \downarrow \mathbf{Fin})[\sI^{-1}]$, the comma category of finite sets with injective totally ordered maps (and therefore all injective maps) inverted. The definitions of a formal diagram and a {\blacktie} diagram are the same as above.

\begin{thm}[Coherence for normal (op)lax symmetric monoidal functors]\label{thm:coherence_normal_symmetric}
Every {\blacktie} diagram of morphisms for a normal lax symmetric monoidal functor commutes.
\end{thm}

In fact, by the same proof as in \cref{localized_comma_cat_is_thin},
\begin{lem}\label{sym_localized_comma_cat_is_thin}
	The localization $(\underline{n} \downarrow \mathbf{Fin})[\sI^{-1}]$ is a thin category.
\end{lem}
Therefore it is only necessary to check the {\punderlying} permutation to see if a diagram is {\blacktie}. In summary, \cref{thm:coherence_normal_symmetric} says that any two parallel formal morphisms inducing the same {\punderlying} permutation of the $\slambda_i$ must agree.

With the modifications above, the proof of \cref{thm:coherence_normal_symmetric} is the same as the proof in \cref{sec:proof_coherence_lax_smc_functors}. We also get that the {\underlying} set functor is an equivalence as in \cref{symmetric_coherence} when the $\slambda_i$ are distinct.

For a strong monoidal functor the necessary modification is to replace ${(\underline{n} \downarrow \mathbf{Fin})[\sI^{-1}]}$ by
\[ {(\underline{n} \downarrow \mathbf{Fin})[\mathbf\Delta^{-1}]} = {(\underline{n} \downarrow \mathbf{Fin})[\mathbf{Fin}^{-1}]} \]
since the maps $m$ are also isomorphisms.

\begin{thm}[Coherence for strong symmetric monoidal functors]\label{thm:coherence_strong_symmetric}
Every {\blacktie} diagram of morphisms for a strong symmetric monoidal functor commutes.
\end{thm}

So any two parallel formal morphisms that induce the same permutation on the $\slambda_i$ must agree. When the $X_i$ are distinct, all formal diagrams commute.

\subsection{Lax shadow functors}\label{sec:lax_shadow}

Let $(\sC,\sC_{Sh})$ and $(\sD,\sD_{Sh})$ be bicategories with shadow. A {\bf lax shadow functor} consists of 
\begin{itemize}
\item a lax functor $\sC \xto{F} \sD$, 
\item a functor on the shadow categories $\sC_{Sh} \xto{H} \sD_{Sh}$, and 
\item shadow commutation maps $s\colon \sh{F(M)} \to H\sh{M}$ for each endomorphism 1-cell $M$ in $\sC$ 
\end{itemize}
such that the diagram in \ref{shadow_lax_only_relation} below commutes. 
We say that $(F,H)$ is {\bf strict} if $F$ is strict and $s$ is an identity map.

Let $\oneShLax$ be the (1-)category whose objects are lax shadow functors of bicategories $\sC \xto{F} \sD$ and whose morphisms are pairs of strict shadow functors forming a strictly commuting square.

There is a forgetful functor 
\begin{equation}\label{eq:lax_shadow_forget}
\oneShLax\to \Graphcat
\end{equation}
that sends $(\sC,\sC_{Sh}) \xto{(F,H)} (\sD,\sD_{Sh})$ to the underlying graph of $\sC$.
The left adjoint of the functor in \eqref{eq:lax_shadow_forget} applied to a graph $\gLambda$ is a lax shadow functor of bicategories
\[(\frbic{\gLambda},\frsbit{\gLambda}) \xto{(\lff{\gLambda},\shlff{\gLambda})}(\lffc{ \gLambda},\shfrsbit{\gLambda}).\]
Here $\frbic{\gLambda}$ is the free bicategory on $\gLambda$ from \cref{sec:coherence_bicat}, $\frsbit{\gLambda}$ is the target shadow category from \cref{sec:shadowed_bicat}, and $\lffc{\gLambda}$ and $\lff{\gLambda}$ are the bicategory and lax functor from \cref{subsec:lax}. The category $\shfrsbit{\gLambda}$ has the following presentation:

\begin{present}\label{rmk:explicit_free_shadow_lax}
	For a graph $\gLambda$, the objects of $\shfrsbit{\gLambda}$ consist of
	\begin{enumerate}[start=1,label={\bfseries O\arabic{objects}}]
		\item\stepcounter{objects}\label{it:sh_out} The objects of $\frsbit{\gLambda}$ with $\shlff{\gLambda}$ written around them, e.g.
		\[ \shlff{\gLambda} \sh{ \ (X_1 \odot X_2) \odot I \ }. \]
		\item\stepcounter{objects}\label{it:functor_out} The endomorphism 1-cells of $\lffc{\gLambda}$ with $\sh{-}$ around them, e.g.
		\[ \sh{ \ \lff{\gLambda}(X_1 \odot X_2) \odot I \odot \lff{\gLambda}(I) \ }. \]
	\end{enumerate}
	The morphisms of $\shfrsbit{\gLambda}$ are generated by the associators \ref{rmk:explicit_free_bicat_1}, unitors \ref{rmk:explicit_free_bicat_2}, and rotators \ref{gen:rotator} 
for the objects in \ref{it:sh_out}, along with the free lax functor generators \ref{rmk:explicit_free_lax_1} to \ref{rmk:explicit_free_lax_6} and the rotators \ref{gen:rotator} for the objects in \ref{it:functor_out}. In addition we have
	\begin{enumerate}[start=1,label={\bfseries G\arabic{generators}}]
		\item\stepcounter{generators}\label{lax_shadow_s}
		Formal shadow commutator maps $s\colon \sh{ \lff{\gLambda} W } \to \shlff{\gLambda} \sh{ W }$.
	\end{enumerate}

The relations are  \ref{rmk:explicit_free_bicat_4}-\ref{rmk:explicit_free_bicat_6} and \ref{fig:shadow_unit_coherence}-\ref{theta:natural} for the objects in \ref{it:sh_out}, 
the relations \ref{eq:pent_tri_lax_functor_all}-\ref{eq:wisk_nat_lax_functor} for the objects in \ref{it:functor_out}, a second copy of the shadow relations \ref{fig:shadow_unit_coherence}-\ref{theta:natural} for the objects in \ref{it:functor_out},
\begin{enumerate}[start=1,label={\bfseries R\arabic{relations}}]
		\item\stepcounter{relations}\label{shadow_lax_only_relation}
		The coherence condition for the shadow commutator
		\[\xymatrix @R=1.5em{
\sh{\lff{\gLambda}(M) \odot  \lff{\gLambda}(N)}\ar[r]^\theta\ar[d]^{\sh{m}}
& \sh{\lff{\gLambda}(N) \odot  \lff{\gLambda}(M)}\ar[d]^{\sh{m}}
\\
\sh{\lff{\gLambda}(M \odot N)}\ar[d]^s
&\sh{\lff{\gLambda}(N\odot M)}\ar[d]^s
\\
\shlff{\gLambda}\sh{M \odot N}\ar[r]^{\shlff{\gLambda}(\theta)}
&\shlff{\gLambda}\sh{N \odot  M}\ ,
}
\]
and
\item\stepcounter{relations}\label{shadow_lax_naturality}
		naturality of $s$ with respect to associators and unitors applied to the word $W$.
	\end{enumerate}
These relations make $\sh{}$ into a shadow from $\lffc{\gLambda}$, $\shlff{\gLambda}$ into a functor from $\frsbit{\gLambda}$, $s$ into a natural transformation, and $(\lff{\gLambda},\shlff{\gLambda})$ into a lax shadow functor.
\end{present}

\begin{lem}
The construction in \cref{eq:lax_functor_supporting_morphism} extends to define a
 {\bf {\underlying} set} functor
 \begin{equation}\label{eq:shadow_underlying_functor}
 U\colon \comp{\shfrsbit{\gLambda}}{W} \to 
(\underline{n} \downarrow \mathbf\Lambda').
\end{equation}
\end{lem}

\begin{proof}
The category $\mathbf\Lambda'$ is  the bi-augmented cyclic category  from \cref{lambda_presentation}. 

The components of $\shfrsbit{\gLambda}$ correspond to lists of cyclically composable edges $X_1,\ldots,X_n$. For each component $\comp{\shfrsbit{\gLambda}}{W}$, $U$ is defined 
as follows:
\begin{itemize}
\item The image of a word $\sh{\lff{\gLambda}(...)\odot...\odot\lff{\gLambda}(...)}$ is the map $\underline{n} \to \underline{k}$ in $\mathbf\Lambda$  composed of 
\begin{itemize}
\item a cyclic permutation of the terms $X_i$ to put them in the desired order, 
\item followed by the map in $\mathbf\Delta$ encoding the grouping of those terms into $\lff{\gLambda}(...)$ blocks as in \cref{eq:lax_functor_supporting_morphism}.
\end{itemize}
\item The image of an object of the form $\shlff{\gLambda}\sh{...}$ is the terminal map $\underline{n} \to \underline{*}$.
\item 
As in \eqref{eq:lax_functor_underlying}, the images of  $m$ and $i$ are codegeneracy and coface maps. 
\item The image of $\theta$ is a  cyclic permutation.
\item The image of  $s$ is the terminal map. 
\item The images of  all other generators (associators and unitors, and rotators inside $\shlff{\gLambda}\sh{...}$) are identity maps.
\end{itemize}
 We then check that the relations go to commuting maps in $\mathbf\Lambda$. The only checks not covered by previous cases are \ref{fig:shadow_unit_coherence}- \ref{theta:natural} for the $\sh{\lff{\gLambda}(...)\odot...\odot\lff{\gLambda}(...)}$ terms, which are straightforward, and \ref{shadow_lax_only_relation}, which commutes by \ref{terminal1}.
\end{proof}

\begin{lem}
The {\punderlying} permutation functor from \cref{lem:underlying_functor_shadow} extends to an {\bf {\punderlying} permutation} functor
\begin{equation}\label{eq:symmetric_underlying_functor_1}
 P\colon \comp{\shfrsbit{\gLambda}}{W} \to 
\sB C_n.
\end{equation}
\end{lem}

\begin{proof}
The functor sends each instance of $\theta$ to the corresponding cyclic permutation of the terms $X_1$, $\ldots$, $X_n$, and all other generators to the identity.
\end{proof}

A {\bf formal diagram} of a lax shadow functor $(\sC,\sC_{Sh}) \xto{(F,H)} (\sD,\sD_{Sh})$ is any diagram in $\sD_{Sh}$ that lifts against the functor $\shfrsbit{\fols{\sC}}\to \sD_{Sh}$.

\begin{defn} 
A formal diagram of morphisms for a lax shadow functor is {\bf \blacktie} if the {\underlying} maps and {\punderlying} permutations for both composites are the same.
\end{defn}

If the terms $\slambda_1$, $\ldots$, $\slambda_n$ are aperiodic then any two formal maps with the same source and target must give the same permutation, so $P$ can be safely ignored.

\begin{thm}[Coherence for (op)lax shadow functors]\label{shadow_not_coherence}
Every {\blacktie} diagram of morphisms for a lax shadow functor commutes.
\end{thm}

\subsubsection{Proof of coherence for lax shadow functors}\label{sec:proof_coherence_lax_shadow_functors}
As in \cref{sec:proof_coherence_lax_smc_functors}, if the {\punderlying} permutations of two formal composites are the same, then those composites can be interpreted as acting on an aperiodic list. Hence, without loss of generality, we can ignore $P$ and focus on the case where the elements $\slambda_i$ are aperiodic.

For each morphism $\alpha\colon \underline{n} \to \underline{k}$ in $\mathbf\Lambda'$ we follow  \cref{ex:define_odot_clique_F} and define a clique
\begin{equation}\label{define_shadow_F_clique}
\sh{ \ \bigodot_{j \in \underline{k}} \lff{\gLambda} \left( \bigodot_{i \in \alpha^{-1}(j)} X_i \right) \ }
\end{equation}
	with maps generated by associators and unitors on both the inside and the outside of the $\lff{\gLambda}$. (We do not include rotators on the outside.)
	
	The terms in each of the inside products $\bigodot_{i \in \alpha^{-1}(j)} X_i$ are arranged using the total ordering on $\underline{\alpha^{-1}(j)}$ inherited from $\underline{n}$ as a \emph{cyclically} ordered set. This order is either of the form $\{i,i+1,\ldots, i+k\}$ with $i+k\leq n$, or $\{i,i+1,...,n,1,2,...,j\}$. When $k = 1$, some care is needed -- the map $\alpha\colon \underline{n} \to \underline{1}$ is given by the data of a partition of $\underline{n}$ into $\{1,...,\ell\} \cup \{\ell+1,...,n\}$, and for this map the induced ordering on $\alpha^{-1}(1)$ is $\{\ell+1,...,n,1,...,\ell\}$. 

	For the terminal morphism $t\colon \underline{n} \to \underline{*}$ we take the clique
\begin{equation}\label{define_shadow_H_clique}
\shlff{\gLambda} \sh{ \ \left( \bigodot_{i \in \underline{n}} X_i \right) \ }
\end{equation}
	as above, except we use associators, unitors, \emph{and} rotators inside the $\shlff{\gLambda}$. Since the $X_i$ are aperiodic, this is a clique by coherence for shadowed bicategories (\cref{shbicat_coherence_formal}).

\begin{lem}[\cref{rmk:plan_of_proofs}\ref{rmk:plan_of_proofs_1}]\label{built_diagram_of_cliques_shadow_functor}
	The cliques \eqref{define_shadow_F_clique} and \eqref{define_shadow_H_clique} extend to a diagram of cliques from $(\underline{n} \downarrow \mathbf\Lambda')$ to $\shfrsbit{\gLambda}$.
\end{lem}

\begin{proof}

Recall from \cref{lambda_presentation} that $\mathbf\Lambda'$ is generated by cofaces, codegeneracies, cycle maps, and a map $t\colon \underline{1} \to \underline{*}$. The coface and codegeneracy maps from $\mathbf\Delta$ give clique maps described in \cref{built_diagram_of_cliques:lem_coherence_bicat}. The cycle maps are assigned to the clique maps that rotate the $\lff{\gLambda}$ terms by one position, using associators, unitors, and rotator maps outside the copies of $\lff{\gLambda}$. Any fixed formula for doing this commutes with the associators and unitors inside the $\lff{\gLambda}$ by naturality (\ref{eq:wisk_nat_lax_functor}), and different formulas agree by coherence for shadowed bicategories (\cref{shbicat_coherence_formal}). 

Finally, composing with the map $t$ applies the shadow commutation $s$. The admissible models are those in the one-fold tensor product $\bigodot_{j \in \underline{1}} \lff{\gLambda}(...)$ that have only the $\lff{\gLambda}(...)$ term and no extra units. This gives a clique map by \ref{shadow_lax_naturality}.

	For the relations from \cref{lambda_presentation} we check that each word in the relation gives the same map of cliques. For \ref{faceface} to \ref{degface3}, this is by the proof of \cref{built_diagram_of_cliques}. For \ref{cycleface}, \ref{cycleface2}, and \ref{cycletorsion} this follows from coherence in a shadowed bicategory (\cref{shbicat_coherence_formal}). \ref{cycledeg} and \ref{cycledeg2} are by naturality of $\theta$ in a shadowed bicategory (\ref{theta:natural}). Finally, the terminal relation \ref{terminal1} follows directly from the coherence \ref{shadow_lax_only_relation}.
\end{proof}

Let $D$ be the diagram of cliques in \cref{built_diagram_of_cliques_shadow_functor}.  
\cref{clique_bijection} defines a functor (\cref{rmk:plan_of_proofs}\ref{rmk:plan_of_proofs_2})
\begin{equation}\label{eq:shadow_lax_iso}  \int_{(\underline{n} \downarrow \mathbf\Lambda')} D \to \comp{\shfrsbit{\gLambda}}{W}.
\end{equation}

\begin{thm}[\cref{rmk:plan_of_proofs}\ref{rmk:plan_of_proofs_3}]\label{same_presentations_shadow}
 	The functor in \eqref{eq:shadow_lax_iso} is an isomorphism of categories.
\end{thm}

\begin{proof}
	By construction, \eqref{eq:shadow_lax_iso}  is a bijection onto the objects of $\comp{\shfrsbit{\gLambda}}{W}$. (If the elements $\slambda_i$ were not distinct then this claim would fail.)

The generators of $\int_{(\underline{n} \downarrow \mathbf\Lambda')} D$ given by \cref{groth_presentation,lambda_presentation} are the generators  in each clique, together with the cofaces, codegeneracies, cycles, and terminal map. These correspond to the generators of $\comp{\shfrsbit{\gLambda}}{W}$ (the cliques giving all expanded instances of $\alpha$, $l$, $r$, and $\theta$, save for $\theta$ on the outside of $\sh{\lff{\gLambda}(...)\odot...\odot\lff{\gLambda}(...)}$, and the horizontal generators giving $i$, $m$, $s$ and the remaining instances of $\theta$).  	Therefore this functor is  surjective on morphisms.
	
The composite functor
\[ \int_{(\underline{n} \downarrow \mathbf\Lambda')} D \xto{\eqref{eq:shadow_lax_iso}}  \comp{\shfrsbit{\gLambda}}{W}
\xto{\eqref{eq:shadow_underlying_functor}} (\underline{n} \downarrow \mathbf\Lambda')\]
is the projection $\pi$ to the base category from \cref{lem:consequences_of_projection_1}.  By that result, $\pi$ is an equivalence of categories and so \eqref{eq:shadow_lax_iso} is faithful.

Since \eqref{eq:shadow_lax_iso} is an isomorphism on objects and full and faithful it is an isomorphism.
\end{proof}

\begin{cor}[\cref{rmk:plan_of_proofs}\ref{rmk:plan_of_proofs_4}]\label{shadow_coherence}
When the elements $\glambda_i$ are aperiodic, the {\underlying} set functor \eqref{eq:shadow_underlying_functor} is an equivalence of categories.
 \end{cor}

This finishes the proof of coherence for lax shadow functors (\cref{shadow_not_coherence}).

\subsubsection{Normal and strong shadow functors}
A lax shadow functor is {\bf normal} if its unit maps $i$ are isomorphisms, and {\bf strong} if it is normal and  the compositions $m$ and shadow commutators $s$ are isomorphisms. As in \cref{subsec:lax}, the coherence theorems for these are proven in the same way as for lax shadow functors.

For normal functors, replace the category $\oneShLax$ with the corresponding category for lax normal functors with a forgetful functor 
\[
\oneNSh\to \Set\]
and let
\[ (\frbic{\gLambda},\frsbit{\gLambda}) \xto{(\nlff{\gLambda},\nshlff{\gLambda})}(\nlffc{ \gLambda},\nshfrsbit{\gLambda}) \]
be the result of applying the free functor to a graph $\gLambda$.   Then the presentation for $\nshfrsbit{\gLambda}$ is as in 
\cref{rmk:explicit_free_shadow_lax} except that the unit maps $i$ are invertible generators. The {\underlying} set functor  goes from a component of $\nshfrsbit{\gLambda}$ to $(\underline{n} \downarrow \mathbf\Lambda')[\sI^{-1}]$, the comma category of the bi-augmented cyclic category in which the injective totally ordered maps (and therefore all injective maps) have been inverted. The definitions of a formal diagram and a {\blacktie} diagram are the same as above.

\begin{thm}[Coherence for normal (op)lax shadow functors]\label{thm:coherence_normal_shadow}
Every {\blacktie} diagram of morphisms for a normal lax shadow functor commutes.
\end{thm}

The proof of \cref{localized_comma_cat_is_thin} is slightly trickier to verify in this case, but it gives
\begin{lem}\label{shad_localized_comma_cat_is_thin}
	The localization $(\underline{n} \downarrow \mathbf\Lambda')[\sI^{-1}]$ is a thin category.
\end{lem}
Therefore it is only necessary to check the {\punderlying} permutation to see if a diagram is {\blacktie}. In summary, \cref{thm:coherence_normal_shadow} says that any two parallel formal morphisms inducing the same {\punderlying} permutation of the $\glambda_i$ must agree.

With the modifications above, the proof is the same as the proof in \cref{sec:proof_coherence_lax_shadow_functors}. We also get that the {\underlying} set functor is an equivalence as in \cref{symmetric_coherence} when the $\slambda_i$ are distinct.

For a strong shadow functor the necessary modification is to replace ${(\underline{n} \downarrow \mathbf\Lambda')[\sI^{-1}]}$ by
\[ {(\underline{n} \downarrow \mathbf\Lambda')[\mathbf\Delta^{-1},t^{-1}]} = {(\underline{n} \downarrow \mathbf\Lambda')[\mathbf\Lambda'^{-1}]} \]
since the maps $m$ and $s$ are also isomorphisms. This category is also thin.

\begin{thm}[Coherence for strong shadow functors]\label{thm:coherence_strong_shadow}
Every {\blacktie} diagram of morphisms for a strong shadow functor commutes.
\end{thm}

So any two parallel formal morphisms that rotate the $\glambda_i$ by the same amount must agree. When the $\glambda_i$ are distinct, or at least aperiodic, all formal diagrams commute.

\bibliographystyle{amsalpha2}
\bibliography{references}%

\vspace{1em}

\end{document}